\documentclass[a4paper]{article} 
\usepackage[english]{babel}
\usepackage[T1]{fontenc}
\usepackage{lmodern}
\usepackage{setspace}
\usepackage[colorlinks=true,allcolors=blue]{hyperref}
\usepackage[a4paper,top=3cm,bottom=2.2cm,left=2cm,right=2cm,marginparwidth=2cm]{geometry}
\usepackage{amsmath,amsfonts,amssymb,amsthm}
\usepackage{graphicx}
\usepackage{natbib}
\usepackage{float}
\usepackage{authblk}
\usepackage{booktabs}
\usepackage{subcaption}
\usepackage{multicol}
\usepackage{xcolor}
\usepackage[colorinlistoftodos]{todonotes}
\usepackage{titlesec}

\newtheorem{theorem}{Theorem}
\newtheorem{example}{Example}

\newcommand{\XRLT}[0]{X^{\text{\scriptsize RLT}}}
\newcommand{\XQ}[0]{X^{\text{\scriptsize Q}}}

\newcommand{\LPQR}[0]{(LP$_{\text{QR}}$)}
\newcommand{\LPQRd}[0]{(LP$_{\text{QR},d}$)}

\title{Degree reduction techniques for polynomial optimization problems}
\author[]{Brais Gonz\'alez-Rodr\'iguez}
\author[]{Joe Naoum-Sawaya\footnote{Corresponding Author}}
\affil[]{Ivey Business School, Western University, 1255 Western Road, London, Ontario, Canada, N6G 4W1.\newline brodriguez@ivey.ca, jnaoumsa@uwaterloo.ca}

\date{}

\begin{document}
\linespread{1.25}
\onehalfspacing

\begin{titlepage}
\maketitle

\begin{abstract}
\noindent {This paper presents a new approach to quadrify a polynomial programming problem, i.e. reduce the polynomial program to a quadratic program, before solving it. The proposed approach, QUAD-RLT, exploits the Reformulation-Linearization Technique (RLT) structure to obtain smaller relaxations that can be solved faster and still provide high quality bounds. QUAD-RLT is compared to other quadrification techniques that have been previously discussed in the literature. The paper presents theoretical as well as computational results showing the advantage of QUAD-RLT compared to other quadrification techniques. Furthermore, rather than quadrifying a polynomial program, QUAD-RLT is generalized to reduce the degree of the polynomial to any degree. Computational results show that reducing the degree of the polynomial to a degree that is higher than two provides computational advantages in certain cases compared to fully quadrifying the problem. Finally, QUAD-RLT along with other quadrification/degree reduction schemes are implemented and made available in the freely available software RAPOSa. 
}
\end{abstract}
\end{titlepage}

\setcounter{page}{2}

\section{Introduction}

Polynomial programming problems are a powerful class of optimization problems that arise in various fields including engineering, economics, and computer science. The ability to efficiently solve such problems has profound implications for decision-making, resource allocation, portfolio optimization, and system design \citep{Cafieri2017, Ghaddar2017, GonzalezDiaz2021}. However, the inherent complexities that are due to the nonlinearity of the problems often render conventional optimization methods ineffective.

In recent years, significant progress has been made to address the computational challenges posed by polynomial programming problems. One promising approach involves the utilization of the Reformulation-Linearization Technique (RLT), \citep{Sherali1992, Sherali1999}. Different implementations of this technique can be found in \cite{Dalkiran2016} and in \cite{González-Rodríguez2023}. Other common approaches for solving polynomial programming problems include the use of semidefinite programming \citep{lasserre11global}, B-Splines \citep{Gawali2017}, linearizations and convexifications techniques \citep{ghaddar2011second, xia2017completely, Elloumi2023}, cutting planes \citep{Gorge2015, ghaddar2016dynamic}, as well as approximate approaches \citep{Li1998, kuang2018completely, kuang2019alternative}. Moreover, many solvers have been developed in recent years to deal with nonlinear programming particularly polynomial programming such as BARON \citep{Tawarmalani2002}, Couenne \citep{Belotti2013}, and ANTIGONE \citep{Misener2014}.

Recently, quadrification techniques in the context of RLT have been proposed to transform a polynomial optimization problem into a quadratic problem \citep{Dalkiran2018}. Quadrification involves the transformation of higher-order polynomials into a series of quadratic expressions while preserving the original problem's feasibility and optimality. By employing quadrification in conjunction with RLT, one can significantly enhance the tractability of polynomial programming problems of a high degree, enabling efficient and reliable solutions even for large-scale instances.

 In this paper, we propose a new quadrification scheme that outperforms the ones presented in \cite{Dalkiran2018}. 
 We also present theoretical results showing the strengths of our proposed approach in terms of {the quality of the resulting bounds and the computational time required to obtain those bounds}. As a middle ground between quadrification and solving a high-degree polynomial, degree reduction transforms the high-degree polynomial into a lower-degree expression while preserving the problem's essential characteristics. By leveraging algebraic manipulations and variable substitutions, the original problem can be transformed into an equivalent but simplified form that is more amenable to common optimization techniques. Thus, in this paper, we also extend all the quadrification techniques of \cite{Dalkiran2018} and our newly introduced quadrification scheme to reduce the degree of a polynomial program, without necessarily quadrifying it. We note that the proposed quadrification and degree reduction schemes can be extended to more general models such as \cite{Borchers1994} and \cite{Wang2011}. Finally, the contributions of this paper extend beyond the theoretical and computational aspects, as we also provide an implementation of the discussed quadrification schemes in a freely available polynomial optimization software, {RAPOSa}.

The remainder of this paper is structured as follows. In {Section~\ref{sec:rltandquad}}, we summarize the main points of the RLT and the quadrification schemes introduced in \cite{Dalkiran2018} and present new theoretical results regarding the obtained linear relaxations. In Section~\ref{sec:newapproach}, we present the proposed approach to quadrify a polynomial optimization problem as well as computational results benchmarking against the approaches of \cite{Dalkiran2018}. In Section~\ref{sec:degreered}, we first extend the quadrification schemes to reduce the degree of a polynomial program and then present computational results evaluating the discussed degree reduction techniques. Finally in Section~\ref{sec:conclusions} we conclude the paper and discuss potential future work.

\section{The Reformulation-Linearization Technique and Quadrification Schemes for Polynomial Problems }\label{sec:rltandquad}

This section presents a summary of the Reformulation-Linearization Technique \citep{Sherali1992} and discusses the quadrification schemes of \cite{Dalkiran2018} along with their impact on the performance of the RLT.

\subsection{The Reformulation-Linearization Technique}\label{sec:rlt}
The Reformulation-Linearization Technique was originally introduced in \cite{Sherali1992}. It was designed to find a global optimum of polynomial optimization problems of the form
\begin{equation}
\begin{array}{rll}
\text{minimize} &  \phi_0(\boldsymbol{x}) &\\
\text{subject to}  &  \phi_r(\boldsymbol{x})\geq \beta_r, &  r=1,\ldots, R_1 \\
&  \phi_r(\boldsymbol{x})=\beta_r, &  r=R_1+1,\ldots,R\\
&  \boldsymbol{x}\in\Omega \subset \mathbb{R}^n\text{,} &
\end{array}
\tag{PP}
\label{eq:PP}
\end{equation}
where $N = \lbrace 1, \dots, n \rbrace$ denotes the set of variables, each $\phi_r(\boldsymbol{x})$ is a polynomial of degree $\delta_r \in \mathbb{N}$ and $\Omega = \lbrace \boldsymbol{x} \in \mathbb{R}^n: 0 \leq l_j \leq x_j \leq u_j < \infty, \, \forall j \in N \rbrace \subset \mathbb{R}^n$ is a hyperrectangle containing the feasible region. The degree of problem \eqref{eq:PP} is $\delta=\max_{r \in \{0,\ldots,R\}} \delta_r$.

In this paper, we associate for each monomial the corresponding multiset containing the indices of the variables (with repetitions) in the monomial. Given a multiset $J$, we define its size $|J|$ as the number of variables in the monomial (including repetitions), which is the degree of the monomial. We also define the multiset $N^\delta$, which is the multiset containing $\delta$ times each element in $N$.

The Reformulation-Linearization Technique can be combined with a branch-and-bound algorithm to solve the linear relaxations of the polynomial problem~\eqref{eq:PP}. These linear relaxations are obtained by replacing each monomial of degree greater than one with a new RLT variable. For instance, if we consider the monomial $x_1x_2x_3$, we replace this monomial with a new RLT variable called $\XRLT_{123}$. More generally, the RLT variables are defined as
\begin{equation}
  \XRLT_J = \prod_{j \in J}x_j, 
  \label{eq:RLTidentity}
\end{equation}
where $J$ is a multiset containing the information about the multiplicity of each variable in the underlying monomial. Then, at each node of the branch-and-bound tree, one would solve the corresponding linear relaxation. Whenever we get a solution of a linear relaxation in which the identities in~\eqref{eq:RLTidentity} hold, we get a feasible solution for~\eqref{eq:PP}. Otherwise, the violations of these identities are used to choose the branching variable.

\cite{Sherali1992} proves the convergence of this branch-and-bound scheme to a global optimum of the polynomial problem~\eqref{eq:PP} if certain constraints, called bound-factor constraints, are added to the linear relaxations. These constraints are of the form
\begin{equation}\label{eq:boundfactor}
F_{\delta}(J_1,J_2)=\prod_{j\in J_1}{(x_j-l_j)}\prod_{j\in J_2}{(u_j-x_j)}\geq 0
\end{equation}
for each pair of multisets $J_1$ and $J_2$ such that $J_1 \cup J_2 \subset N^\delta$ and $|J_1 \cup J_2| = \delta$.

\cite{Sherali2013} shows that it is not necessary to add all bound-factor constraints to the linear relaxations, since certain subsets are enough to ensure the convergence of the algorithm. More precisely, the convergence to a global optimum only requires the inclusion of bound-factor constraints where $J_1 \cup J_2$ is a monomial that appears in~\eqref{eq:PP}, regardless of its degree. Furthermore, convergence is also preserved if, whenever the bound-factor constraints associated with a monomial $J$ are present, all the bound-factor constraints associated with monomials $J'\subset J$ are removed. As a result, $J$-sets are defined as those monomials of degree greater than one present in~\eqref{eq:PP} which are not included in any other monomial (multiset inclusion). In this paper, we denote $\mathbb{J}$ as the set of all the $J$-sets {for a given polynomial problem}.

Given a polynomial $\phi(\boldsymbol{x})$, we denote its linearization by $[\phi(\boldsymbol{x})]_L$, which is the result of replacing every monomial of degree greater than one in $\phi(\boldsymbol{x})$ by the corresponding RLT variable from~\eqref{eq:RLTidentity}. The resulting RLT relaxation solved at the first node of the branch-and-bound tree is then
\begin{equation}
\begin{array}{rll}
\text{minimize} &  [\phi_0(\boldsymbol{x})]_L &\\
\text{subject to}  &  [\phi_r(\boldsymbol{x})]_L\geq \beta_r, & r=1,\ldots, R_1 \\
&  [\phi_r(\boldsymbol{x})]_L=\beta_r, & r=R_1+1,\ldots,R\\
&  [F_{\delta}(J_1,J_2)]_L \geq 0,  & J_1\cup J_2 \in \mathbb{J}\\
&  \boldsymbol{x}\in\Omega \subset \mathbb{R}^n. &
\end{array}
\tag{LP}
\label{eq:LP}
\end{equation}

{Quadratic problems are a particular case of polynomial problems where the degree of the problem is two. Reducing the polynomial problem to a degree two, i.e. quadrifying the problem, along with the RLT can lead to significant computational benefits when solving the problem to optimality \citep{Dalkiran2018}. In the next section, we present the quadrification schemes introduced in \cite{Dalkiran2018} followed by our newly proposed quadrification scheme in Section~\ref{sec:newapproach}.}

\subsection{Quadrification Schemes}\label{sec:quad}

{Quadrification schemes are based on adding new variables to replace the nonlinear terms in the polynomial programming problem, along with new constraints to represent the relations between these new variables and the corresponding nonlinear terms.} \cite{Dalkiran2018} proposed three different schemes to quadrify a polynomial optimization problem of the form~\eqref{eq:PP} and evaluated their impact on the linear relaxations generated by the RLT algorithm. For completeness, this sections provides a summary of these three quadrification schemes.

In the following, we denote with the superscript $Q$ the new variables created by the quadrification and define $\XQ_J$ as 
\begin{equation}
  \XQ_J = \prod_{j \in J}x_j.
  \label{eq:QUADidentity}
\end{equation}

The commonality between the three schemes of \cite{Dalkiran2018} is that each monomial $\prod_{j\in J}x_j$ of degree greater than 2 in~\eqref{eq:PP} is replaced by a new variable $\XQ_J$ while the three schemes differ in the quadratic constraints that are added to ensure the validity of the quadrification {of $\prod_{j\in J}x_j$}. These constraints are summarized next for each of the three schemes.

\begin{description}
\item {\bfseries Scheme 1}. For Scheme 1, the following quadratic constraints are added
\begin{equation}\label{eq:scheme1}
\begin{array}{c}
     \XQ_J=\XQ_{j_1\ldots j_{|J|-1}}x_{j_{|J|}} \\
     \XQ_{j_1\ldots j_{|J|-1}} = \XQ_{j_1\ldots j_{|J|-2}}x_{j_{|J|-1}} \\
     \vdots\\
     \XQ_{j_1j_2}=x_{j_1}x_{j_2}. \\         
\end{array}
\end{equation}
\item {\bfseries Scheme 2}. In this scheme, for each $J^{\prime}\subseteq J$, a new variable $\XQ_{J^\prime}$ is defined and the following quadratic constraints are added
\begin{equation}\label{eq:scheme2}
\XQ_{J^\prime}=\XQ_{J_1}\XQ_{J_2},\text{ for each } J_1, J_2 \text{ such as } J_1\cup J_2 = J^\prime.
\end{equation}
Note that constraints~\eqref{eq:scheme2} are only needed when $J$ is a $J$-set, since these constraints are redundant for any multiset included in a $J$-set.
\item {\bfseries Scheme 3}. In this scheme, for each $J^\prime\in \bigcup_{d=2}^\delta N^d$, a new variable $\XQ_{J^\prime}$ is defined and the following quadratic constraints are added
\begin{equation}\label{eq:scheme3}
\XQ_{J^\prime}=\XQ_{J_1}\XQ_{J_2}, \text{ for each } J_1, J_2 \text{ such as } J_1\cup J_2 = J^\prime.
\end{equation}
\end{description}

{Thus by applying one of the quadrification schemes (Schemes 1--3), a high degree polynomial problem is reduced to a quadratic problem. \cite{Dalkiran2018} also provided a linear relaxation that replaces the quadratic terms resulting from applying Schemes 1--3 by using McCormick linear constraints \citep{McCormick1976}. Note that this is equivalent to creating the RLT relaxation using the $J$-sets of the corresponding quadrified problem since the McCormick constraints are degree-two bound-factor constraints. Subsequently, \cite{Dalkiran2018} proved that the linear relaxation generated by the original RLT (without $J$-sets) is at least as tight as the ones generated by Schemes 1--3. They also showed that the linear relaxation resulting from Scheme~3 is at least as tight as the one resulting from Scheme 2, which is at least as tight as the one resulting from Scheme 1. Finally, they showed that the linear relaxation generated by the RLT with $J$-sets is at least as tight as the ones generated by Schemes 1 and 2}.

While \cite{Dalkiran2018} showed improving computational results on randomly generated polynomial programming instances with degrees between 3 and 10, as we show in Section~\ref{sec:results}, for higher degree polynomials the relaxations become intractable due to their increasing sizes. Next, we present new tightness results for the quadrification schemes of \cite{Dalkiran2018} when $J$-sets are applied. We also provide new results related to the size of the RLT relaxation in terms of the number of constraints and variables which have not been presented before in the literature.

\subsection{New Dimension Results for Different Quadrification Schemes}\label{sec:tight_size_theorems}
To be able to discuss the new dimension results, we first explicitly present the linear relaxations~\eqref{eq:LPsch1}, \eqref{eq:LPsch2}, and \eqref{eq:LPsch3} solved at the root of the branch-and-bound tree after applying $J$-sets with Schemes~1, 2, and 3, respectively. Given a polynomial $\phi(\boldsymbol{x})$, we denote its quadrification by $[\phi(\boldsymbol{x})]_Q$ which is the result of replacing every monomial of degree greater than two present in $\phi(\boldsymbol{x})$ by the corresponding quadratic variable $\XQ_J$. The RLT relaxations after applying each of the schemes are as follows.

\begin{description}
\item {\bfseries RLT relaxation after applying Scheme~1}
\begin{equation}
\begin{array}{rll}
\text{minimize} &  [[\phi_0(\boldsymbol{x})]_Q]_L &\\
\text{subject to}  &  [[\phi_r(\boldsymbol{x})]_Q]_L\geq \beta_r, & 
r=1,\ldots, R_1 \\
&  [[\phi_r(\boldsymbol{x})]_Q]_L=\beta_r, & r=R_1+1,\ldots,R\\
& \left. \hspace{-0.6cm} \begin{array}{rll}
&  \XQ_J=[\XQ_{j_1\ldots j_{|J|-1}}x_{j_{|J|}}]_L \\
&     \XQ_{j_1\ldots j_{|J|-1}} = [\XQ_{j_1\ldots j_{|J|-2}}x_{j_{|J|-1}}]_L \\
&     \vdots\\
&     \XQ_{j_1j_2}=[x_{j_1}x_{j_2}]_L \\  
\end{array}
\right] & |J| > 2, \text{ such that } \prod_{j\in J}x_j \text{ appears in~\eqref{eq:PP}} \\
&  [F_{\delta}(J_1,J_2)]_L \geq 0,  & J_1\cup J_2 \in \mathbb{J}_{Q1}\\
&  \boldsymbol{x}\in\Omega \subset \mathbb{R}^n &
\end{array}
\tag{LP$_1$}
\label{eq:LPsch1}
\end{equation}
where $\mathbb{J}_{Q1}$ is the set with all $J$-sets of the quadratic problem obtained after applying Scheme~1 to \eqref{eq:PP}. That is, all the quadratic terms in \eqref{eq:PP} together with multisets $\{\{j_1\ldots j_{|J|-1}\}, j_{|J|}\}$, $\{\{j_1\ldots j_{|J|-2}\}, j_{|J|-1}\}$, \ldots, $\{j_1,j_2\}$ for each multiset $J$ such that $|J| > 2$ and $\prod_{j\in J}x_j$ appearing in~\eqref{eq:PP}.

\item {\bfseries RLT relaxation after applying Scheme~2}
\begin{equation}
\begin{array}{rll}
\text{minimize} &  [[\phi_0(\boldsymbol{x})]_Q]_L &\\
\text{subject to}  &  [[\phi_r(\boldsymbol{x})]_Q]_L\geq \beta_r, & 
r=1,\ldots, R_1 \\
&  [[\phi_r(\boldsymbol{x})]_Q]_L=\beta_r, & r=R_1+1,\ldots,R\\
& \XQ_{J^\prime} = [\XQ_{J_1} \XQ_{J_2}]_L & J_1\cup J_2=J^\prime\subseteq J \text{ for each }J\in \mathbb{J} \text{ such that } |J| > 2 \\
&  [F_{\delta}(J_1,J_2)]_L \geq 0,  & J_1\cup J_2 \in \mathbb{J}_{Q2}\\
&  \boldsymbol{x}\in\Omega \subset \mathbb{R}^n &
\end{array}
\tag{LP$_2$}
\label{eq:LPsch2}
\end{equation}
where $\mathbb{J}_{Q2}$ is the set with all $J$-sets of the quadratic problem obtained after applying Scheme~2 to \eqref{eq:PP}. That is, all the quadratic terms in \eqref{eq:PP} together with multisets $\{J_1,J_2\}$, for each $J_1\cup J_2\subseteq J$,  with $J\in\mathbb{J}$ and $|J| > 2$.

\item {\bfseries RLT relaxation after applying Scheme~3}
\begin{equation}
\begin{array}{rll}
\text{minimize} &  [[\phi_0(\boldsymbol{x})]_Q]_L &\\
\text{subject to}  &  [[\phi_r(\boldsymbol{x})]_Q]_L\geq \beta_r, & 
r=1,\ldots, R_1 \\
&  [[\phi_r(\boldsymbol{x})]_Q]_L=\beta_r, & r=R_1+1,\ldots,R\\
& \XQ_{J^\prime} = [\XQ_{J_1} \XQ_{J_2}]_L & J_1\cup J_2=J^\prime\in \bigcup_{d=2}^\delta N^d  \\
&  [F_{\delta}(J_1,J_2)]_L \geq 0,  & J_1\cup J_2 \in \mathbb{J}_{Q3}\\
&  \boldsymbol{x}\in\Omega \subset \mathbb{R}^n &
\end{array}
\tag{LP$_3$}
\label{eq:LPsch3}
\end{equation}
where $\mathbb{J}_{Q3}$ is the set with all $J$-sets of the quadratic problem obtained after applying Scheme~3 to \eqref{eq:PP}. That is, all the multisets $\{J_1,J_2\}$, for each $J_1\cup J_2\in\bigcup_{d=2}^\delta N^d$.
\end{description}

{In the following we denote by $S$, $S_1$, $S_2$, and $S_3$, the feasible regions of \eqref{eq:LP}, \eqref{eq:LPsch1}, \eqref{eq:LPsch2}, and \eqref{eq:LPsch3}, respectively. \cite{Dalkiran2018} proved that $S_3\subseteq S_2 \subseteq S_1$ and $S\subseteq S_2$. Next, we {also provide results} about the size of the relaxations.}





\begin{theorem}\label{th:sizerel}
Let $N^{LP}_1$, $N^{LP}_2$, and $N^{LP}_3$ be the number of variables in~\eqref{eq:LPsch1}, \eqref{eq:LPsch2}, and \eqref{eq:LPsch3}, respectively. Let $R^{LP}_1$, $R^{LP}_2$, and $R^{LP}_3$ be the number of constraints in~\eqref{eq:LPsch1}, \eqref{eq:LPsch2}, and \eqref{eq:LPsch3}, respectively. It holds that $N^{LP}_1\leq N^{LP}_2 \leq N^{LP}_3$ and $R^{LP}_1\leq R^{LP}_2 \leq R^{LP}_3$.
\end{theorem}

\begin{proof}
Note that for each multiset $J$, the constraints in~\eqref{eq:scheme1} are included in the constraints in~\eqref{eq:scheme2} (and hence their respective linearizations). Moreover, we also have that $\mathbb{J}_{Q1}\subset \mathbb{J}_{Q2}$. Thus, it follows that $R^{LP}_1\leq R^{LP}_2$. Furthermore, since the objective function does not change between the relaxations, we also have that $N^{LP}_1\leq N^{LP}_2$.

\noindent The same argument can be used to prove $R^{LP}_2\leq R^{LP}_3$ and $N^{LP}_2\leq N^{LP}_3$, since all the constraints in~\eqref{eq:scheme2} are included in~\eqref{eq:scheme3} and $\mathbb{J}_{Q2}\subset \mathbb{J}_{Q3}$.
\end{proof}

\noindent Theorem~\ref{th:sizerel} compares the RLT relaxations resulting from the three quadrification schemes. In the following, we compare these relaxations to \eqref{eq:LP} which is obtained by applying directly the RLT with $J$-sets to~\eqref{eq:PP}.



\begin{theorem}\label{th:sizejsets}
Let $N^{LP}$ be the number of variables in~\eqref{eq:LP}. It holds that $ N^{LP} \leq N^{LP}_2$. 
\end{theorem}

\begin{proof}
{
First, given a multiset $J$ and a positive integer $k$, we denote by $C(J,k)$ the number of possible combinations taking $k$ elements from $J$. The number of these combinations is known when $J$ is a set (without repeated elements) and is given by the combinatorial number ${\binom{|J|}{k}} = \frac{|J|!}{k!(|J|-k)!}$. However, computing this number when $J$ is a multiset requires the solution of a diophantine equation (see \cite{Ehrlich1973} and \cite{Ruskey1996}). Nevertheless, we do not need to compute explicitly the value of that combinatorial number for this proof. In~\eqref{eq:LP}, we need $\sum_{k=2}^{|J|} C(J,k)$ variables to linearize each $J\in\mathbb{J}$. This is because of the fact that, given a multiset $J$, we need one RLT variable for each possible combination between the elements in $J$, in order to linearize the resulting bound-factor constraints. For example, if we consider the monomial $J=\{1,2,3\}$, one bound-factor constraint is $(x_1-l_1)(x_2-l_2)(x_3-l_3)\geq 0$ and when we linearize it, we end up having the variables $X_{123}$, $X_{12}$, $X_{13}$, and $X_{23}$ (the other linearized bound factor constraints have the same variables since they only differ in the lower or upper bound).

For \eqref{eq:LPsch2}, for each $J\in\mathbb{J}$ we need more than $\sum_{k=2}^{|J|} C(J,k)$ variables. In fact, for each $J$, we have constraints of the form~\eqref{eq:scheme2} which require $\sum_{k=2}^{|J|} C(J,k)$ new variables. Furthermore, a new RLT variable $\XRLT_{J_1J_2}$ is needed to linearize each quadratic term that appears in Constraints~\eqref{eq:scheme2}. Therefore, it holds that $ N^{LP} \leq N^{LP}_2$.
}
\end{proof}

{Theorem~\ref{th:sizejsets} shows that~\eqref{eq:LP} does not has {more} variables than~\eqref{eq:LPsch2}. Regarding the number of constraints, we do not have a proof, but the computational results in Appendix~\ref{appendix:results} show that $R^{LP} \leq R^{LP}_2$ holds where $R^{LP}$ is the number of constraints in~\eqref{eq:LP}\footnote{Appendix~\ref{appendix:results} shows aggregated results however we confirm that for each instance we have $R^{LP} \leq R^{LP}_2$ holds.}.}

Regarding comparing~\eqref{eq:LP} to~\eqref{eq:LPsch1}, Example~\ref{ex:cons} shows two polynomial programming problems wherein the first, ~\eqref{eq:LP} has more constraints than \eqref{eq:LPsch1}, while in the second one, \eqref{eq:LP} has less constraints than \eqref{eq:LPsch1}. Similarly, Example~\ref{ex:vars} shows two polynomial programming problems wherein the first, \eqref{eq:LP} has more variables than \eqref{eq:LPsch1}, while in the second one, \eqref{eq:LP} has less variables than \eqref{eq:LPsch1}.

\begin{example}\label{ex:cons}
Consider the following polynomial programming problem:
\begin{equation}
\begin{array}{rll}
\text{minimize} & x_1x_2x_3 &\\
\text{subject to}  &  0\leq l_j\leq x_j\leq u_j, &  r=1,2,3. \\
\end{array}
\tag{Ex$_1$}
\label{eq:PP1}
\end{equation}

\noindent Below, we present the RLT relaxations \eqref{eq:LP1_J} and \eqref{eq:LP1_1} of \eqref{eq:PP1}.

\begin{multicols}{2}
\noindent
\small
\begin{equation}
\begin{array}{rl}
\text{minimize} & \XRLT_{123} \\
\text{subject to}  &  [(x_1-l_1)(x_2-l_2)(x_3-l_3)]_L \geq 0 \\
&  [(x_1-l_1)(x_2-l_2)(u_3-x_3)]_L \geq 0 \\
&  [(x_1-l_1)(u_2-x_2)(x_3-l_3)]_L \geq 0 \\
&  [(x_1-l_1)(u_2-x_2)(u_3-l_3)]_L \geq 0 \\
&  [(u_1-x_1)(x_2-l_2)(x_3-l_3)]_L \geq 0 \\
&  [(u_1-x_1)(x_2-l_2)(u_3-x_3)]_L \geq 0 \\
&  [(u_1-x_1)(u_2-x_2)(x_3-l_3)]_L \geq 0 \\
&  [(u_1-x_1)(u_2-x_2)(u_3-x_3)]_L \geq 0 \\
&  0\leq l_j\leq x_j\leq u_j,~ r=1,2,3. \\
\end{array}
\tag{Ex$_1^{\text{LP}}$}
\label{eq:LP1_J}
\end{equation}
\begin{equation}
\begin{array}{rl}
\text{minimize} & \XQ_{123} \\
\text{subject to} & \XQ_{123}=[\XQ_{12}x_3]_L \\
& \XQ_{12} = [x_1x_2]_L \\
& [(\XQ_{12}-l_1l_2)(x_3-l_3)]_L \geq 0 \\
& [(\XQ_{12}-l_1l_2)(u_3-x_3)]_L \geq 0 \\
& [(u_1u_2-\XQ_{12})(x_3-l_3)]_L \geq 0 \\
& [(u_1u_2-\XQ_{12})(u_3-x_3)]_L \geq 0 \\
& [(x_1-l_1)(x_2-l_2)]_L \geq 0 \\
& [(x_1-l_1)(u_2-x_2)]_L \geq 0 \\
& [(u_1-x_1)(x_2-l_2)]_L \geq 0 \\
& [(u_1-x_1)(u_2-x_2)]_L \geq 0 \\
&  0\leq l_j\leq x_j\leq u_j,~ r=1,2,3. \\
\end{array}
\tag{Ex$_1^{\text{LP}_1}$}
\label{eq:LP1_1}
\end{equation}
\end{multicols}
\noindent The number of constraints in~\eqref{eq:LP1_J} is 8 which is lower than the number of constraints in~\eqref{eq:LP1_1} (10 constraints). Next we consider the following polynomial programming problem:
\begin{equation}
\begin{array}{rll}
\text{minimize} & x_1x_2x_3 + x_1x_2x_4 &\\
\text{subject to}  &  0\leq l_j\leq x_j\leq u_j, &  r=1,2,3,4. \\
\end{array}
\tag{Ex$_2$}
\label{eq:PP2}
\end{equation}

\noindent The RLT relaxations \eqref{eq:LP2_J} and \eqref{eq:LP2_1} of \eqref{eq:PP2} are 
\begin{multicols}{2}
\noindent
\small
\begin{equation}
\begin{array}{rl}
\text{minimize} & \XRLT_{123}+\XRLT_{124} \\
\text{subject to}  &  [(x_1-l_1)(x_2-l_2)(x_3-l_3)]_L \geq 0 \\
&  [(x_1-l_1)(x_2-l_2)(u_3-x_3)]_L \geq 0 \\
&  [(x_1-l_1)(u_2-x_2)(x_3-l_3)]_L \geq 0 \\
&  [(x_1-l_1)(u_2-x_2)(u_3-l_3)]_L \geq 0 \\
&  [(u_1-x_1)(x_2-l_2)(x_3-l_3)]_L \geq 0 \\
&  [(u_1-x_1)(x_2-l_2)(u_3-x_3)]_L \geq 0 \\
&  [(u_1-x_1)(u_2-x_2)(x_3-l_3)]_L \geq 0 \\
&  [(u_1-x_1)(u_2-x_2)(u_3-x_3)]_L \geq 0 \\
&  [(x_1-l_1)(x_2-l_2)(x_4-l_4)]_L \geq 0 \\
&  [(x_1-l_1)(x_2-l_2)(u_4-x_4)]_L \geq 0 \\
&  [(x_1-l_1)(u_2-x_2)(x_4-l_4)]_L \geq 0 \\
&  [(x_1-l_1)(u_2-x_2)(u_4-l_4)]_L \geq 0 \\
&  [(u_1-x_1)(x_2-l_2)(x_4-l_4)]_L \geq 0 \\
&  [(u_1-x_1)(x_2-l_2)(u_4-x_4)]_L \geq 0 \\
&  [(u_1-x_1)(u_2-x_2)(x_4-l_4)]_L \geq 0 \\
&  [(u_1-x_1)(u_2-x_2)(u_4-x_4)]_L \geq 0 \\
&  0\leq l_j\leq x_j\leq u_j,~ r=1,2,3,4, \\
\end{array}
\tag{Ex$_2^{\text{LP}}$}
\label{eq:LP2_J}
\end{equation}
\begin{equation}
\begin{array}{rl}
\text{minimize} & \XQ_{123}+\XQ_{124} \\
\text{subject to} & \XQ_{123}=[\XQ_{12}x_3]_L \\
& \XQ_{12} = [x_1x_2]_L \\
& \XQ_{124} = [\XQ_{12}x_4]_L \\
& [(\XQ_{12}-l_1l_2)(x_3-l_3)]_L \geq 0 \\
& [(\XQ_{12}-l_1l_2)(u_3-x_3)]_L \geq 0 \\
& [(u_1u_2-\XQ_{12})(x_3-l_3)]_L \geq 0 \\
& [(u_1u_2-\XQ_{12})(u_3-x_3)]_L \geq 0 \\
& [(x_1-l_1)(x_2-l_2)]_L \geq 0 \\
& [(x_1-l_1)(u_2-x_2)]_L \geq 0 \\
& [(u_1-x_1)(x_2-l_2)]_L \geq 0 \\
& [(u_1-x_1)(u_2-x_2)]_L \geq 0 \\
& [(\XQ_{12}-l_1l_2)(x_4-l_4)]_L \geq 0 \\
& [(\XQ_{12}-l_1l_2)(u_4-x_4)]_L \geq 0 \\
& [(u_1u_2-\XQ_{12})(x_4-l_4)]_L \geq 0 \\
& [(u_1u_2-\XQ_{12})(u_4-x_4)]_L \geq 0 \\
&  0\leq l_j\leq x_j\leq u_j,~ r=1,2,3,4. \\
\end{array}
\tag{Ex$_2^{\text{LP}_1}$}
\label{eq:LP2_1}
\end{equation}
\end{multicols}

\noindent The number of constraints in~\eqref{eq:LP2_J} is 16 which is larger than the number of constraints in~\eqref{eq:LP2_1} (15 constraints).
\end{example}

\begin{example}\label{ex:vars}
Consider the polynomial programming problem~\eqref{eq:PP1} and the relaxations \eqref{eq:LP1_J} and \eqref{eq:LP1_1}. We have 7 variables in \eqref{eq:LP1_J}: $x_1$, $x_2$, $x_3$, $\XRLT_{12}$, $\XRLT_{13}$, $\XRLT_{23}$, and $\XRLT_{123}$ while we have 6 variables in \eqref{eq:LP1_1}: $x_1$, $x_2$, $x_3$, $\XQ_{12}$, $[\XQ_{12}x_3]_L$, and $\XQ_{123}$.

Consider the following polynomial programming problem:

\begin{equation}
\begin{array}{rll}
\text{minimize} & x_1x_2x_3 +x_1x_2+x_1x_3+x_2x_3 &\\
\text{subject to}  &  0\leq l_j\leq x_j\leq u_j, &  r=1,2,3. \\
\end{array}
\tag{Ex$_3$}
\label{eq:PP4}
\end{equation}
The corresponding RLT relaxations \eqref{eq:LP4_J} and \eqref{eq:LP4_1} of \eqref{eq:PP4} are

\begin{multicols}{2}
\noindent
\small
\begin{equation}
\begin{array}{rl}
\text{minimize} & \XRLT_{123}+\XRLT_{12}+\XRLT_{13}+\XRLT_{23} \\
\text{subject to}  &  [(x_1-l_1)(x_2-l_2)(x_3-l_3)]_L \geq 0 \\
&  [(x_1-l_1)(x_2-l_2)(u_3-x_3)]_L \geq 0 \\
&  [(x_1-l_1)(u_2-x_2)(x_3-l_3)]_L \geq 0 \\
&  [(x_1-l_1)(u_2-x_2)(u_3-l_3)]_L \geq 0 \\
&  [(u_1-x_1)(x_2-l_2)(x_3-l_3)]_L \geq 0 \\
&  [(u_1-x_1)(x_2-l_2)(u_3-x_3)]_L \geq 0 \\
&  [(u_1-x_1)(u_2-x_2)(x_3-l_3)]_L \geq 0 \\
&  [(u_1-x_1)(u_2-x_2)(u_3-x_3)]_L \geq 0 \\
&  0\leq l_j\leq x_j\leq u_j,~ r=1,2,3, \\
\end{array}
\tag{Ex$_3^{\text{LP}}$}
\label{eq:LP4_J}
\end{equation}
\begin{equation}
\begin{array}{rl}
\text{minimize} & \XQ_{123}+\XRLT_{12}+\XRLT_{13}+\XRLT_{23} \\
\text{subject to} & \XQ_{123}=[\XQ_{12}x_3]_L \\
& \XQ_{12} = [x_1x_2]_L \\
& [(\XQ_{12}-l_1l_2)(x_3-l_3)]_L \geq 0 \\
& [(\XQ_{12}-l_1l_2)(u_3-x_3)]_L \geq 0 \\
& [(u_1u_2-\XQ_{12})(x_3-l_3)]_L \geq 0 \\
& [(u_1u_2-\XQ_{12})(u_3-x_3)]_L \geq 0 \\
& [(x_1-l_1)(x_2-l_2)]_L \geq 0 \\
& [(x_1-l_1)(u_2-x_2)]_L \geq 0 \\
& [(u_1-x_1)(x_2-l_2)]_L \geq 0 \\
& [(u_1-x_1)(u_2-x_2)]_L \geq 0 \\
& [(x_1-l_1)(x_3-l_3)]_L \geq 0 \\
& [(x_1-l_1)(u_3-x_3)]_L \geq 0 \\
& [(u_1-x_1)(x_3-l_3)]_L \geq 0 \\
& [(u_1-x_1)(u_3-x_3)]_L \geq 0 \\
& [(x_2-l_2)(x_3-l_3)]_L \geq 0 \\
& [(x_2-l_2)(u_3-x_3)]_L \geq 0 \\
& [(u_2-x_2)(x_3-l_3)]_L \geq 0 \\
& [(u_2-x_2)(u_3-x_3)]_L \geq 0 \\
&  0\leq l_j\leq x_j\leq u_j,~ r=1,2,3. \\
\end{array}
\tag{Ex$_3^{\text{LP}_1}$}
\label{eq:LP4_1}
\end{equation}
\end{multicols}
We have 7 variables in \eqref{eq:LP4_J}: $x_1$, $x_2$, $x_3$, $\XRLT_{12}$, $\XRLT_{13}$, $\XRLT_{23}$, and $\XRLT_{123}$ while we have 9 variables in \eqref{eq:LP4_1}: $x_1$, $x_2$, $x_3$, $\XRLT_{12}$, $\XRLT_{13}$, $\XRLT_{23}$, $\XQ_{12}$, $[\XQ_{12}x_3]_L$, and $\XQ_{123}$.
\end{example}

Examples~\ref{ex:cons} and \ref{ex:vars} show that there is no strict relationship between the sizes of \eqref{eq:LP} and~\eqref{eq:LPsch1}. Next we show that under certain loose conditions, relaxation \eqref{eq:LPsch1} generated by Scheme 1 is smaller than \eqref{eq:LP}. Particularly, Theorem~\ref{th:sizerel2} provided next shows that \eqref{eq:LPsch1} is smaller than \eqref{eq:LP} for polynomial problems that contain one monomial with $|J| \geq 4$ and $J$ is a set that does not contain variable repetition.



\begin{theorem}\label{th:sizerel2}
Consider a polynomial program~\eqref{eq:PP} with only one monomial $\prod_{j\in J}x_j$ with a degree greater than one where $J$  is a set without repeated variables. Let us also assume for simplicity, that all the variables of the problem appear in at least one monomial of degree one. It holds that, if $|J| \geq 4$, then  $N^{LP}_1 \leq N^{LP}$ and $R^{LP}_1 \leq R^{LP}$.
\end{theorem}

\begin{proof}
The number of constraints in~\eqref{eq:LP} is $R$ (number of constraints in~\eqref{eq:PP}) plus the number of bound-factor constraints. Since the only monomial of degree greater than one is $\prod_{j\in J}x_j$, then the number of bound-factor constraints is $2^{|J|}$. Moreover, the number of variables in~\eqref{eq:LP} is $n$ (number of variables of~\eqref{eq:PP}) plus the number of RLT variables that appear in the bound-factor constraints, which is $\sum_{k=2}^{|J|}\frac{|J|!}{k!(|J|-k)!}$.

Regarding~\eqref{eq:LPsch1}, the number of constraints is $R$, plus the number of constraints from~\eqref{eq:scheme1} (which is $|J|-1$), plus the number of bound-factor constraints. For each quadratic term that appears in the problem after applying Scheme~1, we have 4 bound-factor constraints and hence, the number of constraints in~\eqref{eq:scheme1} is $R+|J|-1+4(|J|-1)=R+5(|J|-1)$.
Regarding the number of variables, we have the variables in~\eqref{eq:PP} (which is $n$) plus the number of variables generated by Scheme~1 (which is $|J|-1$) and the number of RLT variables generated, which is the number of quadratic terms in the problem after applying Scheme~1 (which is $|J|-1$). Therefore, the total number of variables in~\eqref{eq:LPsch1} is $n+2(|J|-1)$.

Thus, we need to prove that, for $|J|\geq 4$, it holds that $5(|J|-1) \leq 2^{|J|}$ and $2(|J|-1)\leq \sum_{k=2}^{|J|}\frac{|J|!}{k!(|J|-k)!}$. Solving numerically the inequality $5(|J|-1) \leq 2^{|J|}$ results in $|J|\leq 1.61$ or $|J|\geq 3.82$. Therefore, if $|J|\geq 4$, then $R^{LP}_1 \leq R^{LP}$. Furthermore, for $k=2$, we have $\frac{|J|!}{k!(|J|-k)!}=\frac{|J|(|J|-1)}{2}$. Rearranging the inequality $2(|J|-1)\leq \frac{|J|(|J|-1)}{2}$, we get $|J|\leq 1$ or $|J|\geq 4$. Therefore, for $|J|\geq 4$, we have $2(|J|-1)\leq \frac{|J|(|J|-1)}{2}\leq \sum_{k=2}^{|J|}\frac{|J|!}{k!(|J|-k)!}$ and hence, $R^{LP}_1 \leq R^{LP}$.
\end{proof}

{Despite the fact that Theorem~\ref{th:sizerel2} refers to a specific case with only one monomial, the intuition behind the result hints that this result may happen often in general. In fact, the computational experiments in Appendix~\ref{appendix:results} confirm that $N^{LP}_1 \leq N^{LP}$ and $R^{LP}_1 \leq R^{LP}$ occurs frequently (98\% of the tested cases).}

{Given the results of~\cite{Dalkiran2018} and Theorems~\ref{th:sizerel}--\ref{th:sizerel2}, we conclude the following}. If we want good lower bounds along the branch-and-bound tree, the best strategy is to use the standard RLT with $J$-sets, since the linear relaxation~\eqref{eq:LP} is tighter. However, in order to have manageable relaxations that can be solved faster than standard RLT with $J$-sets, Scheme 1 with $J$-sets appears to be the better approach. In fact, the computational results shown in Section~\ref{sec:results} show that {this trade-off between the quality of the bounds and the computational complexity of the relaxations depends on the degree of the problem, with Scheme 1 being favored when the degree of the problem is higher.}

In the following section, we present a new quadrification scheme that outperforms the ones discussed in \cite{Dalkiran2018}, particularly for high degree polynomial problems.

\section{A new quadrification scheme}\label{sec:newapproach}
In this section, we present a new quadrification scheme which we refer to as QUAD-RLT. The main objective of this new scheme is to keep the good properties of Scheme 1 in terms of the size of the linear relaxation but take into account how the RLT relaxation is built.

First, we introduce the following notation. Given two multisets $J$ and $J^\prime$ such that $J^\prime\subseteq J$, we denote by $J^\prime_J$ the complementary multiset, that is, the multiset with the elements in $J$ after removing the ones in $J^\prime$. As an example to clarify, if $J=\{1,1,1,2,3,4,4,4\}$ and $J^\prime=\{1,2,4,4\}$, then $J^\prime_J=\{1,1,3,4\}$. Next, we present the new quadrification scheme QUAD-RLT which can be summarized with the following steps:
\begin{itemize}
    \item Step 1. Consider all the monomials of degree greater than 2 in~\eqref{eq:PP} and sort them from highest to lowest degree. We denote this ordered set by $H$. We define another set $G$ with all the monomials of degree $2$ in~\eqref{eq:PP}.
    \item Step 2. Select and delete the first monomial $\prod_{j\in J}x_j$ from $H$.
    \item Step 3. Select from $H\cup G$ the monomial $\prod_{j\in J^\prime}x_j$ with the highest degree such that $J^\prime\subset J$. If there exists such a monomial, go to Step 4a. Otherwise, go to Step 4b.
    \item Step 4a. First, replace the monomial $\prod_{j\in J}x_j$ with a new variable $X_J$. Given the complementary multiset~$J^\prime_J$, we denote by $J^\prime_J[i]$ the 
$i$-th element of $J^\prime_J$ and by $J^\prime_J[1:i]$ the first $i$ elements of the multiset $J^\prime_J$ (if $i=0$, then $J^\prime_J[1:i]$ is the empty set). Then, for each $k\in\{|J^\prime_J|,\ldots,1\}$  add the following quadratic constraint to~\eqref{eq:PP}
    \begin{equation}
        \XQ_{J^\prime,J^\prime_J[1:k]}=\XQ_{J^\prime,J^\prime_J[1:k-1]}\cdot x_{J^\prime_J[k]}
    \end{equation}
and go to Step 5. 
    \item Step 4b. Replace the monomial $\prod_{j\in J}x_j$ with a new variable $\XQ_J$ and add the set of quadratic constraints~\eqref{eq:scheme1}. Then go to Step 5.
    
    \item Step 5. If $H$ is not empty, go to Step 2; otherwise, stop.
\end{itemize}

Let  \LPQR~denote the relaxation resulting from applying QUAD-RLT. Next, we present Example~\ref{ex:quadrlt} to illustrate the QUAD-RLT scheme and particularly highlight the differences between the resulting relaxation \LPQR~compared to the relaxation resulting from applying Scheme~1. 
\begin{example}\label{ex:quadrlt}
Consider the following polynomial programming problem~\eqref{eq:PP5}
\begin{equation}
\begin{array}{rll}
\text{minimize} & x_1x_3 -x_1x_2x_3 &\\
\text{subject to}  &  0\leq l_j\leq x_j\leq u_j, &  r=1,2,3 .\\
\end{array}
\tag{Ex$_4$}
\label{eq:PP5}
\end{equation}
The corresponding relaxation \LPQR~after applying QUAD-RLT and the relaxation~\eqref{eq:LPsch1} after applying Scheme~1 are presented next and are denoted by by~\eqref{eq:LP5QR} and~\eqref{eq:LP5_1}, respectively.

\begin{multicols}{2}
\noindent
\small
\begin{equation}
\begin{array}{rll}
\text{minimize} & \XRLT_{13} - \XQ_{123} &\\
\text{subject to} & \XQ_{123}=[\XQ_{13}x_2]_L \\
& \XQ_{13} = \XRLT_{13} \\
& [(\XQ_{13}-l_1l_3)(x_2-l_2)]_L \geq 0 \\
& [(\XQ_{13}-l_1l_3)(u_2-x_2)]_L \geq 0 \\
& [(u_1u_3-\XQ_{13})(x_2-l_2)]_L \geq 0 \\
& [(u_1u_3-\XQ_{13})(u_2-x_2)]_L \geq 0 \\
& [(x_1-l_1)(x_3-l_3)]_L \geq 0 \\
& [(x_1-l_1)(u_3-x_3)]_L \geq 0 \\
& [(u_1-x_1)(x_3-l_3)]_L \geq 0 \\
& [(u_1-x_1)(u_3-x_3)]_L \geq 0. \\
\end{array}
\tag{Ex$_4^{\text{LP}_\text{QR}}$}
\label{eq:LP5QR}
\end{equation}
\begin{equation}
\begin{array}{rll}
\text{minimize} & \XRLT_{13} - \XQ_{123} &\\
\text{subject to} & \XQ_{123}=[\XQ_{12}x_3]_L \\
& \XQ_{12} = \XRLT_{12} \\
& [(\XQ_{12}-l_1l_2)(x_3-l_3)]_L \geq 0 \\
& [(\XQ_{12}-l_1l_2)(u_3-x_3)]_L \geq 0 \\
& [(u_1u_2-\XQ_{12})(x_3-l_3)]_L \geq 0 \\
& [(u_1u_2-\XQ_{12})(u_3-x_3)]_L \geq 0 \\
& [(x_1-l_1)(x_2-l_2)]_L \geq 0 \\
& [(x_1-l_1)(u_2-x_2)]_L \geq 0 \\
& [(u_1-x_1)(x_2-l_2)]_L \geq 0 \\
& [(u_1-x_1)(u_2-x_2)]_L \geq 0 \\
& [(x_1-l_1)(x_3-l_3)]_L \geq 0 \\
& [(x_1-l_1)(u_3-x_3)]_L \geq 0 \\
& [(u_1-x_1)(x_3-l_3)]_L \geq 0 \\
& [(u_1-x_1)(u_3-x_3)]_L \geq 0. \\
\end{array}
\tag{Ex$_4^{\text{LP}_1}$}
\label{eq:LP5_1}
\end{equation}
\end{multicols}
\end{example}

\eqref{eq:LP5QR} in Example~\ref{ex:quadrlt} has less constraints and less variables than \eqref{eq:LP5_1} (4~constraints and 1~variable less). Theorem~\ref{th:quadrlt} generalizes this to any polynomial programming problem.

\begin{theorem}\label{th:quadrlt}
Let $N^{LP}_{QR}$ and $R^{LP}_{QR}$ be the number of variables and constraints, respectively, in \LPQR. It holds that $N^{LP}_{QR}\leq N^{LP}_{1}$ and $R^{LP}_{QR}\leq R^{LP}_{1}$.
\end{theorem}

\begin{proof}
Regarding the number of constraints in the relaxation, we have two types of constraints, the constraints used for quadrifying and the bound-factor constraints. For the constraints used for quadrifying, it is clear that QUAD-RLT adds less constraints than Scheme~1 since they differ in Step~4a in which QUAD-RLT adds $|J^\prime_J|$ constraints instead of $|J|-1$. Regarding the bound-factor constraints, since QUAD-RLT adds less constraints for quadrifying, it also adds less quadratic terms and hence, less bound-factor constraints. Thus, $R^{LP}_{QR}\leq R^{LP}_{1}$.

Finally, since QUAD-RLT adds less constraints for quadrifying, it also has less $\XQ$ variables and less quadratic terms, hence less RLT variables. Thus, $N^{LP}_{QR}\leq N^{LP}_{1}$.
\end{proof}

Theorem~\ref{th:quadrlt} shows that the relaxations created with QUAD-RLT are smaller (less variables and constraints) than those created by Scheme~1. Generally, this leads to relaxations that are solved faster. Regarding the tightness of the relaxation, it would have been ideal if QUAD-RLT also provided a tighter relaxation than Scheme~1 however Example~\ref{ex:quadtigh} presented next shows that the resulting relaxation \LPQR~may or may not be tighter than \eqref{eq:LPsch1}. 

\begin{example}\label{ex:quadtigh}
Consider the following polynomial programming problem \eqref{eq:PP6}
\begin{equation}
\begin{array}{rll}
\text{minimize} & x_1x_2x_3x_4 -10x_1x_2-x_1x_3x_4 &\\
\text{subject to}  &  1\leq x_1\leq 2, 9\leq x_2\leq 10, 1\leq x_3\leq 2, 9\leq x_4\leq 10. & \\
\end{array}
\tag{Ex$_5$}
\label{eq:PP6}
\end{equation}
The corresponding relaxation \LPQR~after applying QUAD-RLT and the relaxation~\eqref{eq:LPsch1} after applying Scheme~1 are presented next and are denoted by~\eqref{eq:LP6_QR} and~\eqref{eq:LP6_1}, respectively.
\begin{multicols}{2}
\noindent
\small
\begin{equation}
\begin{array}{rll}
\text{minimize} & \XQ_{1234} -10 \XRLT_{12}-\XQ_{134} &\\
\text{subject to} & \XQ_{1234}=[\XQ_{134}x_2]_L \\
& \XQ_{134}=[\XQ_{13}x_4]_L \\
& \XQ_{13} = \XRLT_{13} \\
& [(\XQ_{134}-9)(x_2-9)]_L \geq 0 \\
& [(\XQ_{134}-9)(10-x_2)]_L \geq 0 \\
& [40-\XQ_{134})(x_2-9)]_L \geq 0 \\
& [(40-\XQ_{134})(10-x_2)]_L \geq 0 \\
& [(\XQ_{13}-1)(x_4-9)]_L \geq 0 \\
& [(\XQ_{13}-1)(10-x_4)]_L \geq 0 \\
& [(4-\XQ_{13})(x_4-9)]_L \geq 0 \\
& [(4-\XQ_{13})(10-x_4)]_L \geq 0 \\
& [(x_1-1)(x_2-9)]_L \geq 0 \\
& [(x_1-1)(10-x_2)]_L \geq 0 \\
& [(2-x_1)(x_2-9)]_L \geq 0 \\
& [(2-x_1)(10-x_2)]_L \geq 0 \\
& [(x_1-1)(x_3-1)]_L \geq 0 \\
& [(x_1-1)(2-x_3)]_L \geq 0 \\
& [(2-x_1)(x_3-1)]_L \geq 0 \\
& [(2-x_1)(2-x_3)]_L \geq 0. \\
\end{array}
\tag{Ex$_5^{\text{LP}_{\text{QR}}}$}
\label{eq:LP6_QR}
\end{equation}
\begin{equation}
\begin{array}{rll}
\text{minimize} & \XQ_{1234} -10 \XRLT_{12}-\XQ_{134} &\\
\text{subject to} & \XQ_{1234}=[\XQ_{123}x_4]_L \\
& \XQ_{123}=[\XQ_{12}x_3]_L \\
& \XQ_{12} = \XRLT_{12} \\
& \XQ_{134}=[\XQ_{13}x_4]_L \\
& \XQ_{13} = \XRLT_{13} \\
& [(\XQ_{123}-9)(x_4-9)]_L \geq 0 \\
& [(\XQ_{123}-9)(10-x_4)]_L \geq 0 \\
& [40-\XQ_{123})(x_4-9)]_L \geq 0 \\
& [(40-\XQ_{123})(10-x_4)]_L \geq 0 \\
& [(\XQ_{12}-9)(x_3-1)]_L \geq 0 \\
& [(\XQ_{12}-9)(2-x_3)]_L \geq 0 \\
& [(20-\XQ_{12})(x_3-1)]_L \geq 0 \\
& [(20-\XQ_{12})(2-x_3)]_L \geq 0 \\
& [(\XQ_{13}-1)(x_4-9)]_L \geq 0 \\
& [(\XQ_{13}-1)(10-x_4)]_L \geq 0 \\
& [(4-\XQ_{13})(x_4-9)]_L \geq 0 \\
& [(4-\XQ_{13})(10-x_4)]_L \geq 0 \\
& [(x_1-1)(x_2-9)]_L \geq 0 \\
& [(x_1-1)(10-x_2)]_L \geq 0 \\
& [(2-x_1)(x_2-9)]_L \geq 0 \\
& [(2-x_1)(10-x_2)]_L \geq 0 \\
& [(x_1-1)(x_3-1)]_L \geq 0 \\
& [(x_1-1)(2-x_3)]_L \geq 0 \\
& [(2-x_1)(x_3-1)]_L \geq 0 \\
& [(2-x_1)(2-x_3)]_L \geq 0. \\
\end{array}
\tag{Ex$_5^{\text{LP}_1}$}
\label{eq:LP6_1}
\end{equation}
\end{multicols}
\noindent Solving the linear programs \eqref{eq:LP6_QR} and \eqref{eq:LP6_1}, we get the optimal solutions $-43.81$ and $-38$, respectively and thus \eqref{eq:LP6_1} is tighter than \eqref{eq:LP6_QR}. Next, consider the polynomial program \eqref{eq:PP5} with $l_1=l_2=l_3=0$ and $u_1=u_2=u_3=1$. The optimal solutions of~\eqref{eq:LP5QR}  and ~\eqref{eq:LP5_1} are $0$ and $-0.5$, respectively, and thus, in this case \eqref{eq:LP5QR} is tighter than \eqref{eq:LP5_1}.
\end{example}

Example~\ref{ex:quadtigh} shows that there is no particular relation between the tightness of \LPQR ~and \eqref{eq:LPsch1}, for a given polynomial optimization problem~\eqref{eq:PP}. However, in the construction of Example~\ref{ex:quadtigh}, we found it difficult to find a problem in which \eqref{eq:LPsch1} is tighter than \LPQR. In fact, the results provided next in Section~\ref{sec:results}, show that for randomly generated problems, in the majority of cases \LPQR~is tighter than \eqref{eq:LPsch1} ({96\%} of the tested cases 
in the computational results).



\subsection{Computational results}\label{sec:results}
This section presents a detailed computational analysis of the different quadrification schemes summarized in Section~\ref{sec:quad} and the newly introduced quadrification scheme QUAD-RLT presented in Section~\ref{sec:newapproach}. For that, we have implemented all the schemes in RAPOSa~\citep{González-Rodríguez2023}, a solver for solving polynomial programming problems based on the RLT technique discussed in Section~\ref{sec:rlt}. RAPOSa is freely available on \url{https://raposa.usc.es} and all the quadrification schemes discussed in this paper including QUAD-RLT are now available in RAPOSa starting version {4.1.0}.

To generate the test instances for the computational analysis, we used all the instances with degree~2 from \cite{Dalkiran2016} 
(30 instances in total with different densities) and {we increase their degree up to a degree $\delta$, with more or less monomials depending on a given parameter $k$, according to the following:}
 
{
\begin{itemize}
    \item For each $i\in\{2,\ldots,\delta\}$, create a random multiset $J$, with $|J|=i$, containing indices of variables present in the problem.
    \item For each random multiset $J$ created in the previous step, add the corresponding monomial $\prod_{j\in J}x_j$ to the objective function, with a random coefficient between $-10$ and $10$.
    \item Repeat the previous steps $k$ times.
\end{itemize}
}

The resulting instances are split into $8$ different sets with $30$ instances in each one where each set is characterised by a degree $\delta\in\{5,10,15,20\}$ and {parameter} $k\in\{1,10\}$. 
The computations are performed on the supercomputer Finisterrae~III, provided by Galicia Supercomputing Centre (CESGA) using a 32 cores Intel Xeon Ice Lake 8352Y CPU, 256GB of RAM, and a 1TB SSD. The stopping criteria for RAPOSa are a  time limit of 1 hour and a relative gap of $1\mathrm{e}{-3}$.

\begin{table}{\scriptsize
\begin{subtable}[t]{0.49\textwidth}
\centering
\begin{tabular}{|l|rrrrr|}
\toprule
    & Bas. & Sch. 1 & Sch. 2 & Sch. 3 & Q-RLT \\
\midrule
gmean gap & 0.002 & 0.003 & 0.003 & NA & 0.003\\
gmean time & 12.9 & 30.4 & 23.7 & 3600.0 & 20.78 \\
solved & 23 & 21 & 21 & 0 & 21 \\
\bottomrule
\end{tabular}
\caption{$\delta=5$, $k=1$}
\label{table:quadd5k1}
\end{subtable}
\hspace{\fill}
\begin{subtable}[t]{0.49\textwidth}
\centering
\begin{tabular}{|l|rrrrr|}
\toprule
    & Bas. & Sch. 1 & Sch. 2 & Sch. 3 & Q-RLT \\
\midrule
gmean gap & 0.003 & 0.018 & 0.005 & NA & 0.012 \\
gmean time & 81.1 & 774.2 & 340.9 & 3600.0 & 574.0\\
solved & 22 & 8 & 18 & 0 & 10 \\
\bottomrule
\end{tabular}
\caption{$\delta=5$, $k=10$}
\label{table:quadd5k10}
\end{subtable}

\vspace{0.5cm}
\begin{subtable}[t]{0.49\textwidth}
\centering
\begin{tabular}{|l|rrrrr|}
\toprule
    & Bas. & Sch. 1 & Sch. 2 & Sch. 3 & Q-RLT \\
\midrule
gmean gap & 0.016 & 0.071 & 0.015 & NA & 0.019 \\
gmean time & 611.4 & 406.4 & 1464.2 & 3600.0 & 301.4 \\
solved & 11 & 8 & 12 & 0 & 10 \\
\bottomrule
\end{tabular}
\caption{$\delta=10$, $k=1$}
\label{table:quadd10k1}
\end{subtable}
\hspace{\fill}
\begin{subtable}[t]{0.49\textwidth}
\centering
\begin{tabular}{|l|rrrrr|}
\toprule
    & Bas. & Sch. 1 & Sch. 2 & Sch. 3 & Q-RLT \\
\midrule
gmean gap & 0.297 & 0.260 & NA & NA & 0.090 \\
gmean time & 2373.6 & 1655.4 & 3600.0 & 3600.0 & 796.8 \\
solved & 6 & 4 & 0 & 0 & 7 \\
\bottomrule
\end{tabular}
\caption{$\delta=10$, $k=10$}
\label{table:quadd10k10}
\end{subtable}

\vspace{0.5cm}
\begin{subtable}[t]{0.49\textwidth}
\centering
\begin{tabular}{|l|rrrrr|}
\toprule
    & Bas. & Sch. 1 & Sch. 2 & Sch. 3 & Q-RLT \\
\midrule
gmean gap & NA & 3.000 & NA & NA & 0.488 \\
gmean time & 3600.0 & 1855.8 & 3600.0 & 3600.0 & 675.9 \\
solved & 0 & 2 & 0 & 0 & 6 \\
\bottomrule
\end{tabular}
\caption{$\delta=15$, $k=1$}
\label{table:quadd15k1}
\end{subtable}
\hspace{\fill}
\begin{subtable}[t]{0.49\textwidth}
\centering
\begin{tabular}{|l|rrrrr|}
\toprule
    & Bas. & Sch. 1 & Sch. 2 & Sch. 3 & Q-RLT \\
\midrule
gmean gap & NA & 1.316 & NA & NA & 0.287 \\
gmean time & 3600.0 & 3600.0 & 3600.0 & 3600.0 & 2054.5 \\
solved & 0 & 0 & 0 & 0 & 5 \\
\bottomrule
\end{tabular}
\caption{$\delta=15$, $k=10$}
\label{table:quadd15k10}
\end{subtable}

\vspace{0.5cm}
\begin{subtable}[t]{0.49\textwidth}
\centering
\begin{tabular}{|l|rrrrr|}
\toprule
    & Bas. & Sch. 1 & Sch. 2 & Sch. 3 & Q-RLT \\
\midrule
gmean gap & NA & 8.219 & NA & NA & 1.069 \\
gmean time & 3600.0 & 1072.9 & 3600.0 & 3600.0 & 416.2 \\
solved & 0 & 4 & 0 & 0 & 7 \\
\bottomrule
\end{tabular}
\caption{$\delta=20$, $k=1$}
\label{table:quadd20k1}
\end{subtable}
\hspace{\fill}
\begin{subtable}[t]{0.49\textwidth}
\centering
\begin{tabular}{|l|rrrrr|}
\toprule
    & Bas. & Sch. 1 & Sch. 2 & Sch. 3 & Q-RLT \\
\midrule
gmean gap & NA & 6.606 & NA & NA & 1.485 \\
gmean time & 3600.0 & 3600.0 & 3600.0 & 3600.0 & 2166.7 \\
solved & 0 & 0 & 0 & 0 & 3 \\
\bottomrule
\end{tabular}
\caption{$\delta=20$, $k=10$}
\label{table:quadd20k10}
\end{subtable}
\caption{Quadrification results for different $\delta$ and $k$.}
\label{table:quad}
}\end{table}

Table~\ref{table:quad} summarizes all the results for each value of $\delta$ and $k$. The results show the geometric mean of the optimality gap, the geometric mean of the running time (in seconds), and the number of solved instances within the time limit. The gap is computed as (upper bound - lower bound)/|upper bound|. If an instance is not solved within the time limit, the time that is included in the computation of the mean is 3600 seconds. {It is important to highlight that using the geometric mean makes the differences between the schemes seem smaller, but it is the standard when comparing the performance of different algorithms.} 
The following setups are tested:
\begin{itemize}
    \item Bas.: is the Baseline which uses the RLT with $J$-sets to solve the original polynomial program~\eqref{eq:PP}.
    \item Sch. 1: applies Scheme 1 to quadrify~\eqref{eq:PP} and solves the resulting problem with RLT and $J$-sets.
    \item Sch. 2: applies Scheme 2 to quadrify~\eqref{eq:PP} and solves the resulting problem with RLT and $J$-sets.
    \item Sch. 3: applies Scheme 3 to quadrify~\eqref{eq:PP} and solves the resulting problem with RLT and $J$-sets.
    \item Q-RLT: applies QUAD-RLT to quadrify~\eqref{eq:PP} and solves the resulting problem with RLT and $J$-sets.
\end{itemize}

As discussed  in Theorem~\ref{th:sizerel}, the linear relaxations \eqref{eq:LPsch3} created with Scheme 3 are too large and computationally challenging to solve. As seen in Table~\ref{table:quad}, even the first linear relaxation could not be solved for any of the instances within the 1~hour time limit using Scheme~3. 
Similarly for Scheme 2, the problems become challenging to solve as the degree of the polynomial problem becomes higher (Tables~\ref{table:quadd10k10}--\ref{table:quadd20k10}). 
The Baseline version eventually failed to solve the first relaxation for all the instances in Tables~\ref{table:quadd15k1}--\ref{table:quadd20k10}.

For problems with a low degree polynomials, i.e. $\delta=5$ (Tables~\ref{table:quadd5k1} and \ref{table:quadd5k10}), the Baseline version achieves the best performance compared to all the quadrification schemes in terms of mean gap, mean solution time, and the number of problems solved within the time limit. However, as the degree of the polynomial increases to~10 (Tables~\ref{table:quadd10k1} and \ref{table:quadd10k10}), the newly proposed approach QUAD-RLT starts to outperform the other quadrification schemes and the Baseline. In particular, for the case where $k=1$, Table~\ref{table:quadd10k1} shows that QUAD-RLT outperforms all the quadrification schemes in terms of computational time (25\% less computational time compared to Scheme 1 which is the second fastest). In terms of gap, both the Baseline and Scheme 2 reach a better average gap and solve more instances (11 and 12 for the Baseline and Scheme 2 respectively, compared to 10 solved with QUAD-RLT). The advantage of QUAD-RLT becomes more evident as $k$ is increased to~10 (Table~\ref{table:quadd10k10}). In that case, QUAD-RLT achieves better results for all the metrics compared to the Baseline and the other quadrification schemes. Eventually as the degree of polynomial is increased to 15 and 20 (Tables~\ref{table:quadd15k1}--\ref{table:quadd20k10}), QUAD-RLT continues to be the best approach and for the hardest instances with $\delta=20$ and $k=10$, QUAD-RLT is still able to solve 3 instances in an average of 2166.7 seconds while the other approaches failed to solve any of the instances within the time limit.


In summary, the results in Table~\ref{table:quad} show that for low degree polynomials, using the Baseline scheme is the better approach compared to the other quadrification schemes, however as the degree of the polynomial increases, QUAD-RLT becomes the better quadrification scheme that outperforms the Baseline and the other schemes that have been previously introduced in the literature. Besides the computational results presented in Table~\ref{table:quad}, in Appendix~\ref{appendix:results}, we present results evaluating the theoretical tightness and problem size guarantees discussed in Section~\ref{sec:quad} and Section~\ref{sec:newapproach}.


\section{Degree reduction}\label{sec:degreered}

The results in Section~\ref{sec:results} showed that for problems with degree 5 (Tables~\ref{table:quadd5k1} and \ref{table:quadd5k10}), the Baseline approach achieves the best results. This observation also aligns with the results shown in \cite{Dalkiran2018} where for low degree polynomials (lower than 7), solving directly the polynomial programming problem using $J$-sets, i.e. the Baseline approach without quadrification, achieves the best performance. Thus, instead of reducing the degree of the polynomial to 2, i.e. quadrification, it might be best to reduce the degree of the polynomial to an intermediate degree that is higher than 2 and then solving the resulting problem. This section thus evaluates the use of QUAD-RLT and Schemes 1--3 for degree reduction rather than quadrification.

In the following, we generalize Schemes 1--3 and QUAD-RLT to reduce the degree of~\eqref{eq:PP} to a certain degree $d\geq 2$. Given a polynomial $\phi(\boldsymbol{x})$ and a degree $d\geq 2$, we denote its ``reduction to degree~$d$'' by $[\phi(\boldsymbol{x})]_{D_d}$ which is the result of replacing every monomial of degree greater than $d$ present in $\phi(\boldsymbol{x})$ by the corresponding variable $\XQ_J$. {Moreover, given a multiset $J=\{j_1,\ldots,j_{|J|}\}$ and a degree $d\geq2$, we define the following set of constraints~$C_J$:
\begin{itemize}
\item Step 1: Let $J^\prime=J$.
\item Step 2a: Add to $C_J$ the constraint $\XQ_{J^\prime}=[\XQ_{j^\prime_1,\ldots j^\prime_{|J^\prime|-d+1}}x_{j^\prime_{|J^\prime|-d+2}}\cdots x_{j^\prime_{|J^\prime|}}]_L$. Update $J^\prime$ by deleting elements $j^\prime_{|J^\prime|-d+2}, \ldots j^\prime_{|J^\prime|}$, and go to Step 3.
\item Step 2b: Add to $C_J$ the constraint $\XQ_{J^\prime}=[x_{j^\prime_1}\cdots x_{j^\prime_{|J^\prime|}}]_L$ and stop.
\item Step 3. If $|J^\prime| > d$, go to step 2a. Otherwise, go to step 2b.
\end{itemize}
}
\noindent For a given degree $d\geq 2$, the corresponding linear relaxation generated by the RLT after applying each of the schemes to reduce the degree of~\eqref{eq:PP} to $d$ is as follows.

\begin{enumerate}
\item {\bfseries Scheme~1:}
\begin{equation}
\hspace{-0.3cm}
\begin{array}{rll}
\text{minimize} &  [[\phi_0(\boldsymbol{x})]_{D_d}]_L &\\
\text{subject to}  &  [[\phi_r(\boldsymbol{x})]_{D_d}]_L\geq \beta_r, & 
r=1,\ldots, R_1 \\
&  [[\phi_r(\boldsymbol{x})]_{D_d}]_L=\beta_r, & r=R_1+1,\ldots,R\\
&  {C_J}, & \text{for each }|J| > d, \text{ such that } \prod_{j\in J}x_j \text{ appears in~\eqref{eq:PP}}\\

&  [F_{\delta}(J_1,J_2)]_L \geq 0,  & J_1\cup J_2 \in \mathbb{J}_{D_{d}1}\\
&  \boldsymbol{x}\in\Omega \subset \mathbb{R}^n &
\end{array}
\tag{LP$_{1,d}$}
\label{eq:LPdsch1}
\end{equation}
where $\mathbb{J}_{D_{d}1}$ is the set with all the $J$-sets of the problem obtained after reducing the degree of \eqref{eq:PP} to $d$ using Scheme~1. Thus, we consider the multisets associated with all the terms with degree greater or equal to $2$ and lower or equal to $d$, present in~\eqref{eq:PP}, together with the multisets associated with the monomials that appear in $C_J$ for each multiset $J$ such that $|J| > d$ and $\prod_{j\in J}x_j$ appearing in~\eqref{eq:PP}. 
Note, that \eqref{eq:LPdsch1} is equivalent to \eqref{eq:LPsch1} when $d=2$.

\item {\bfseries Scheme~2:}
\begin{equation}
\begin{array}{rll}
\text{minimize} &  [[\phi_0(\boldsymbol{x})]_{D_d}]_L &\\
\text{subject to}  &  [[\phi_r(\boldsymbol{x})]_{D_d}]_L\geq \beta_r, & 
r=1,\ldots, R_1 \\
&  [[\phi_r(\boldsymbol{x})]_{D_d}]_L=\beta_r, & r=R_1+1,\ldots,R\\
& \XQ_{J^\prime} = [\prod_{d^\prime=1}^{d} \XQ_{J_{d^\prime}}]_L & \bigcup_{d^\prime=1}^{d}J_{d^\prime}=J^\prime\subseteq J \text{ for each }J\in \mathbb{J} \text{ such that } |J|>d \\
&  [F_{\delta}(J_1,J_2)]_L \geq 0,  & J_1\cup J_2 \in \mathbb{J}_{D_{d}2}\\
&  \boldsymbol{x}\in\Omega \subset \mathbb{R}^n &
\end{array}
\tag{LP$_{2,d}$}
\label{eq:LPdsch2}
\end{equation}
where $\mathbb{J}_{D_{d}2}$ is the set with all the $J$-sets of the problem obtained after reducing the degree of \eqref{eq:PP} to~$d$ using Scheme~2. Thus, we consider the multisets associated with all the terms with degree greater or equal to $2$ and lower or equal to $d$, present in~\eqref{eq:PP}, together with the multisets $\{J_{d^\prime}\}_{d^\prime=1}^d$, for each $\bigcup_{d^\prime=1}^{d}J_{d^\prime}\subseteq J$,  with $J\in\mathbb{J}$ and $|J|>d$.
Note, that \eqref{eq:LPdsch2} is equivalent to \eqref{eq:LPsch2} when $d=2$.

\item {\bfseries Scheme~3:}
\begin{equation}
\begin{array}{rll}
\text{minimize} &  [[\phi_0(\boldsymbol{x})]_{D_d}]_L &\\
\text{subject to}  &  [[\phi_r(\boldsymbol{x})]_{D_d}]_L\geq \beta_r, & 
r=1,\ldots, R_1 \\
&  [[\phi_r(\boldsymbol{x})]_{D_d}]_L=\beta_r, & r=R_1+1,\ldots,R\\
& \XQ_{J^\prime} = [\prod_{d^\prime=1}^{d} \XQ_{J_{d^\prime}}]_L & \bigcup_{d^\prime=1}^{d}J_{d^\prime}\in \bigcup_{d^\prime=d}^\delta N^{d^\prime}  \\
&  [F_{\delta}(J_1,J_2)]_L \geq 0,  & J_1\cup J_2 \in \mathbb{J}_{D_{d}3}\\
&  \boldsymbol{x}\in\Omega \subset \mathbb{R}^n &
\end{array}
\tag{LP$_{3,d}$}
\label{eq:LPdsch3}
\end{equation}
where $\mathbb{J}_{D_{d}3}$ is the set with all the $J$-sets of the problem obtained after reducing the degree of \eqref{eq:PP} to~$d$ using Scheme~3. Thus, we consider the multisets associated with all the terms with degree greater or equal to $2$ and lower or equal to $d$, present in~\eqref{eq:PP}, together with the multisets $\{J_{d^\prime}\}_{d^\prime=1}^d$, for each  $\bigcup_{d^\prime=1}^{d}J_{d^\prime} \in\bigcup_{d^\prime=d}^\delta N^{d^\prime}$. 
Note, that \eqref{eq:LPdsch3} is equivalent to \eqref{eq:LPsch3} when $d=2$.

\item {\bfseries QUAD-RLT:} QUAD-RLT is applied to reduce the degree of a polynomial to degree $d$ as follows:
\begin{itemize}
    \item Step 1. Consider all the monomials of degree greater than $d$ present in~\eqref{eq:PP} and sort them from highest to lowest degree. We denote this ordered set by $H$. Define another set $G$ with all the monomials in~\eqref{eq:PP} of degree lower or equal to $d$ and greater or equal to 2.
    \item Step 2. Select and delete the first monomial $\prod_{j\in J}x_j$ from $H$.
    \item Step 3. Select from $H\cup G$ the monomial with the highest degree such that $J^\prime\subset J$. If there exists such a monomial, go to Step 4a. Otherwise, go to Step 4b.
    {

\item Step 4a. First, replace the monomial $\prod_{j\in J}x_j$ with a new variable $\XQ_J$. Next consider the complementary multiset $J^\prime_J$ and do the following:

\begin{itemize}
\item Step i: Define $J^{\prime\prime}=J^\prime\cup J^\prime_J$. If $|J^\prime_J| \geq d$ go to Step ii.a. Otherwise go to Step ii.b.
\item Step ii.a: Add the constraint $\XQ_{J^{\prime\prime}}=[\XQ_{j^{\prime\prime}_1,\ldots j^{\prime\prime}_{|J^{\prime\prime}|-d+1}}x_{j^{\prime\prime}_{|J^{\prime\prime}|-d+2}}\cdots x_{j^{\prime\prime}_{|J^{\prime\prime}|}}]_L$, update $J^{\prime\prime}$ by deleting elements $j^{\prime\prime}_{|J^{\prime\prime}|-d+2}, \ldots j^{\prime\prime}_{|J^{\prime\prime}|}$, and go to Step iii.
\item Step ii.b: Add the constraint $\XQ_{J^{\prime\prime}}=[\XQ_{j^{\prime\prime}_1\ldots j^{\prime\prime}_{|J^\prime|}}x_{j^{\prime\prime}_{|J^\prime|+1}}\cdots x_{j^{\prime\prime}_{|J^{\prime\prime}|}}]_L$ and stop.
\item Step iii. If $|J^{\prime\prime}|-|J^\prime| \geq d$, go to Step ii.a. Otherwise, go to Step ii.b.
\end{itemize}

    \item Step 4b. Replace the monomial $\prod_{j\in J}x_j$ with a new variable $\XQ_J$ and add the set of constraints defined by $C_J$. Then, go to Step 5.    }
    \item Step 5. If $H$ is not empty, go to Step 2; otherwise, stop.
\end{itemize}
\end{enumerate}
\noindent Let \LPQRd~denote the relaxation resulting from applying QUAD-RLT degree reduction to \eqref{eq:PP}. We note that \LPQRd~with $d=2$ is equivalent to \LPQR.
Next in Theorem~\ref{th:degreered}, we show that the theoretical results presented for quadrification in Section~\ref{sec:quad} and Section~\ref{sec:newapproach} still hold for degree reduction. 

\begin{theorem}\label{th:degreered}
Let $S$, $S_{1,d}$, $S_{2,d}$, $S_{3,d}$, and $S_{QR,d}$ be the feasible regions of~\eqref{eq:LP}, \eqref{eq:LPdsch1}, \eqref{eq:LPdsch2}, \eqref{eq:LPdsch3}, and \LPQRd, respectively. Let $N^{LP}$, $N^{LP}_{1,d}$, $N^{LP}_{2,d}$, $N^{LP}_{3,d}$ and $N^{LP}_{QR,d}$ be the number of variables in~\eqref{eq:LP}, \eqref{eq:LPdsch1}, \eqref{eq:LPdsch2}, \eqref{eq:LPdsch3}, and \LPQRd, respectively. Let $R^{LP}$, $R^{LP}_{1,d}$, $R^{LP}_{2,d}$, $R^{LP}_{3,d}$ and $R^{LP}_{QR,d}$ be the number of constraints in~\eqref{eq:LP}, \eqref{eq:LPdsch1}, \eqref{eq:LPdsch2}, \eqref{eq:LPdsch3}, and \LPQRd, respectively. For a given degree $d\geq2$, the following holds
\begin{enumerate}
    \item $S_{3,d}\subseteq S_{2,d} \subseteq S_{1,d}$,
    \item $S\subseteq S_{2,d}$,
    \item $N^{LP}_{QR,d}\leq N^{LP}_{1,d}\leq N^{LP}_{2,d}\leq N^{LP}_{3,d}$,
    \item $N^{LP}\leq N^{LP}_{2,d}\leq N^{LP}_{3,d}$,
    \item $R^{LP}_{QR,d}\leq R^{LP}_{1,d}\leq R^{LP}_{2,d}\leq R^{LP}_{3,d}$.
\end{enumerate}
\end{theorem}

\noindent The proof of Theorem~\ref{th:degreered} can be easily derived from  the theoretical results presented in Section~\ref{sec:quad} and Section~\ref{sec:newapproach}. Moreover, similar examples to Examples~\ref{ex:cons},~\ref{ex:vars},~\ref{ex:quadrlt}, and \ref{ex:quadtigh} can be found.

\subsection{Computational results}


\begin{table}[!htbp]{\small
\begin{subtable}[t]{\textwidth}
\centering
\begin{tabular}{|l|rrrrrrrrr|}
\toprule
    & Bas. & Sch. 1 & Q-RLT & Sch. 1 & Q-RLT & Sch. 1 & Q-RLT & Sch. 1 & Q-RLT \\
    & & ($d=8$) & ($d=8$) & ($d=6$) & ($d=6$) & ($d=4$) & ($d=4$) & ($d=2$) & ($d=2$)\\
\midrule
gmean gap & 0.016 & 0.027 & 0.008  & 0.024 & 0.008 & 0.035 & 0.014 & 0.071 & 0.019   \\
gmean time & 611.4 & 851.7 & 347.1& 339.0 & 170.3 & 424.6 & 259.2 & 406.4 & 301.4   \\
solved & 11 & 11 & 15 & 11 & 13 & 10 & 10 & 8 & 10   \\
\bottomrule
\end{tabular}
\caption{$\delta=10$, $k=1$}
\label{table:reducedd10k1}
\end{subtable}

\vspace{0.25cm}
\begin{subtable}[t]{\textwidth}
\centering
\begin{tabular}{|l|rrrrrrrrr|}
\toprule
    & Bas. & Sch. 1 & Q-RLT & Sch. 1 & Q-RLT & Sch. 1 & Q-RLT & Sch. 1 & Q-RLT \\
    & & ($d=8$) & ($d=8$) & ($d=6$) & ($d=6$) & ($d=4$) & ($d=4$) & ($d=2$) & ($d=2$)\\
\midrule
gmean gap & 0.297 & 0.329 & 0.140 & 0.170 & 0.055 & 0.253 & 0.080 & 0.260 & 0.090   \\
gmean time & 2373.6 & 2468.6 & 1797.1 & 1723.0 & 1052.0 & 1567.8 & 856.1 & 1655.5 & 796.8   \\
solved & 6 & 5 & 5 & 5 & 7 & 4 & 7 & 4 & 7   \\
\bottomrule
\end{tabular}
\caption{$\delta=10$, $k=10$}
\label{table:reducedd10k10}
\end{subtable}

\vspace{0.25cm}
\begin{subtable}[t]{\textwidth}
\centering
\begin{tabular}{|l|rrrrrrrrr|}
\toprule
    & Bas. & Sch. 1 & Q-RLT & Sch. 1 & Q-RLT & Sch. 1 & Q-RLT & Sch. 1 & Q-RLT \\
    & & ($d=8$) & ($d=8$) & ($d=6$) & ($d=6$) & ($d=4$) & ($d=4$) & ($d=2$) & ($d=2$)\\
\midrule
gmean gap & NA & 2.667 & 0.471 & 3.135 & 0.470 & 3.035 & 0.483 & 3.000 & 0.488   \\
gmean time & 3600.0 & 2223.0 & 1042.8 & 1938.6 & 692.5 & 2006.3 & 681.6 & 1855.8 & 675.9   \\
solved & 0 & 2 & 6 & 2 & 6 & 2 & 6 & 2 & 6   \\
\bottomrule
\end{tabular}
\caption{$\delta=15$, $k=1$}
\label{table:reducedd15k1}
\end{subtable}

\vspace{0.25cm}
\begin{subtable}[t]{\textwidth}
\centering
\begin{tabular}{|l|rrrrrrrrr|}
\toprule
    & Bas. & Sch. 1 & Q-RLT & Sch. 1 & Q-RLT & Sch. 1 & Q-RLT & Sch. 1 & Q-RLT \\
    & & ($d=8$) & ($d=8$) & ($d=6$) & ($d=6$) & ($d=4$) & ($d=4$) & ($d=2$) & ($d=2$)\\
\midrule
gmean gap & NA & 5.547 & 0.834 & 1.947 & 0.347 & 1.383 & 0.276 & 1.316 & 0.287   \\
gmean time & 3600.0 & 3600.0 & 3600.0 & 3600.0 & 2428.7 & 3600.0 & 2123.3 & 3600.0 & 2054.5   \\
solved & 0 & 0 & 0 & 0 & 4 & 0 & 5 & 0 & 5   \\
\bottomrule
\end{tabular}
\caption{$\delta=15$, $k=10$}
\label{table:reducedd15k10}
\end{subtable}

\vspace{0.25cm}
\begin{subtable}[t]{\textwidth}
\centering
\begin{tabular}{|l|rrrrrrrrr|}
\toprule
    & Bas. & Sch. 1 & Q-RLT & Sch. 1 & Q-RLT & Sch. 1 & Q-RLT & Sch. 1 & Q-RLT \\
    & & ($d=8$) & ($d=8$) & ($d=6$) & ($d=6$) & ($d=4$) & ($d=4$) & ($d=2$) & ($d=2$)\\
\midrule
gmean gap & NA & 15.485 & 1.484 & 10.669 & 1.081 & 8.472 & 1.070 & 8.219 & 1.069   \\
gmean time & 3600.0 & 1510.9 & 773.5 & 1083.0 & 427.8 & 1164.9 & 411.0 & 1072.9 & 416.2   \\
solved & 0 & 4 & 6 & 4 & 7 & 4 & 7 & 4 & 7   \\
\bottomrule
\end{tabular}
\caption{$\delta=20$, $k=1$}
\label{table:reducedd20k1}
\end{subtable}

\vspace{0.25cm}
\begin{subtable}[t]{\textwidth}
\centering
\begin{tabular}{|l|rrrrrrrrr|}
\toprule
    & Bas. & Sch. 1 & Q-RLT & Sch. 1 & Q-RLT & Sch. 1 & Q-RLT & Sch. 1 & Q-RLT \\
    & & ($d=8$) & ($d=8$) & ($d=6$) & ($d=6$) & ($d=4$) & ($d=4$) & ($d=2$) & ($d=2$)\\
\midrule
gmean gap & NA & NA & 4.762 & 10.081 & 1.831 & 6.760 & 1.481 & 6.606 & 1.485   \\
gmean time & 3600.0 & 3600.0 & 3600.0 & 3600.0 & 2967.8 & 3600.0 & 2206.7 & 3600.0 & 2166.7   \\
solved & 0 & 0 & 0 & 0 & 3 & 0 & 3 & 0 & 3   \\
\bottomrule
\end{tabular}
\caption{$\delta=20$, $k=10$}
\label{table:reducedd20k10}
\end{subtable}
\caption{Degree reduction results for different $\delta$ and $k$.}
\label{table:reduced}
}\end{table}

In this section, we present a detailed computational analysis of the different schemes to reduce the degree of a polynomial problem \eqref{eq:PP}. We use the same computational setup as in Section~\ref{sec:results}. We also present results on the same set of instances used in Section~\ref{sec:results} except for instances with degree $\delta=5$ since, as discussed earlier, for these instances the best approach is to use Baseline without degree reduction. 

Table~\ref{table:reduced} reports results for Baseline, Scheme 1, and QUAD-RLT since these schemes outperform Schemes 2 and 3 (see Section~\ref{sec:results}). For Scheme 1 and QUAD-RLT we present results for degree reduction to {$d=8$, $d=6$, $d=4$, and $d=2$}. We note that we completed more computational experiments with higher values of $d$ but these experiments revealed that it is best to reduce the degree below 8 rather than higher values.

{Table~\ref{table:reducedd10k1} summarizes the performance of the different approaches on problems with degree $\delta=10$ and $k=1$. Looking at the number of solved instances, we see that the best approach is to use QUAD-RLT to reduce the degree to 8. Regarding the optimality gap, it is also best to use QUAD-RLT and reduce the degree to either 8 or~6. Moreover, in terms of running time, we see that the best option is to use QUAD-RLT and to reduce the degree to~6. 
Similar results are also observed for problems with degree $\delta=10$ and $k=10$ shown in Table~\ref{table:reducedd10k10}. QUAD-RLT is still the best approach with the degree reduced to 6 (the best one in terms of gap and solved instances) or to 4 (the best one in terms of running time and solved instances). Comparing with the results in Table~\ref{table:reducedd10k1}, we notice that as the size of the problem increased, reducing the degree of the problem further, improved the computational time.

Table~\ref{table:reducedd15k1} presents results when the degree of the original problem is 15. Comparing with the results in Table~\ref{table:reducedd10k1} where the degree of the original problem was 10, we see that it is best to reduce the degree to 2 with QUAD-RLT to get the lowest computational time. However, in terms of optimality gap, the best option is to reduce the degree to 6. We note that, QUAD-RLT  solved 6 instances, while Scheme 1 was only able to solve 2 and the Baseline did not solve any instance within the time limit. Finally, the remaining Tables~\ref{table:reducedd15k10}, \ref{table:reducedd20k1} and~\ref{table:reducedd20k10} show the same observations where it is best to lower the degree of the problem further as $\delta$ and $k$ increase. In all cases, QUAD-RLT is better than the Baseline and Scheme 1. 
}



{In summary, Table~\ref{table:reduced} shows that reducing the degree of the polynomial program to a degree higher than~2 is beneficial and it is not necessary best to quadrify the problem, i.e. reduce the degree to 2. For instance when the degree of the problem is relatively small ($\delta=10$), the best option is to reduce the problem to degree $8$ or $6$, since we can handle the resulting RLT relaxations. However, when the degree of the problem is larger ($\delta=15$ and $\delta=20$), it is best to reduce to degree 4 or 2 (quadrify) in order to have RLT relaxations that can be solved. Furthermore, we note that the difference between reducing to one degree or another is bigger when Scheme 1 is used, while they are closer when QUAD-RLT is used. This is because the impact of QUAD-RLT on the performance is more significant than Scheme~1 as discussed in Section~\ref{sec:results}. Thus QUAD-RLT is less sensitive to the degree chosen for reduction.}

\clearpage
\section{Conclusion}\label{sec:conclusions}
{This paper presented a new quadrification approach, QUAD-RLT, to solve polynomial programs using the well known Reformulation-Linearization Technique. QUAD-RLT is specifically designed to exploit the structure of the RLT and thus builds reformulations that are more computationally tractable compared to other quadrification techniques that have been previously discussed in the literature. In particular, we compared QUAD-RLT to three quadrification schemes that have been recently presented and showed that QUAD-RLT produces smaller reformulations and is able to achieve better computational performance in terms of computational time and bounds particularly for high degree polynomial programs. We also presented theoretical results related to the tightness of the relaxations and their size when applying the different quadrification techniques.

As discussed in the paper, reducing the degree of the polynomial program to 2 may not necessarily lead to the best computational performance. Thus, we also generalized the various quadrification techniques including QUAD-RLT to reduce the degree of a polynomial to a degree that is higher than 2. The resulting degree reduction methods are evaluated and the computational results showed that QUAD-RLT outperformed the other methods. Finally, we note that all the methods discussed in this paper have been implemented in the freely available polynomial optimization software, {RAPOSa}.

In future work, we plan to explore the use of machine learning techniques to potentially identify the best degree to which a polynomial program should be reduced to in order to achieve the best computational performance. Furthermore, as these methods are now available in a generic polynomial programming solver, we aim to explore potential industry applications which are modeled using high degree polynomial programs.
}

\section*{Acknowledgements}

This research was supported by NSERC Discovery Grant RGPIN-201703962.

\bibliographystyle{ecta}
\bibliography{main}

\newpage

\appendix

\section{Results evaluating tightness and size at the root node}\label{appendix:results}

In this appendix, we present results evaluating the tightness of the bounds and the size of the linear relaxation at the root node, after applying the different quadrification schemes. The results shown in Table~\ref{table:appendix} provide the following:
\begin{itemize}
    \item ``Opt'': optimal value of the linear relaxation at the root node.
    \item ``Nvars'': number of variables in the linear relaxation at the root node.
    \item ``Ncons'': number of constraints in the linear relaxation at the root node.
\end{itemize}

{
The ``\% Diff.'' is the percentage change in the average of the corresponding value (optimal value, number of variables or number of constraints) with respect to the Baseline. In the case where the Baseline is not capable of solving the problem (indicated by ``NA'' in the Baseline results), the percentage change is calculated with respect to Scheme 1. We note that a negative percentage in the ``Opt'' value implies a worse performance since the bound at the root node is worse. However, a negative percentage in ``Nvars'' and ``Ncons'' implies a better performance since the relaxation has fewer number of variables or fewer constraints, respectively. The results also show the number of instances for which each scheme is strictly best (indicated by ``Count'').}

As expected following the theoretical results presented in Sections~\ref{sec:tight_size_theorems} and~\ref{sec:newapproach}, QUAD-RLT is superior to the other schemes for the majority of cases and in particular when the degree of the polynomial is high ($\delta=15$ and $\delta=20$). For a low degree polynomial, i.e. $\delta=5$ and $\delta=10$, the Baseline often provides the better bound while QUAD-RLT leads to a smaller size problem in the majority of cases. However as the degree of the polynomial increases to $\delta=15$ and $\delta = 20$, the Baseline cannot solve the problems and QUAD-RLT achieves the better bounds in the majority of cases and has smaller size problems compared to all other schemes. We also note that as expected and as discussed in Section~\ref{sec:quad}, Schemes 2 and 3 are always worse than Scheme 1 and QUAD-RLT, and are only able to handle problems with low degree $\delta=5$.

\begin{table}[!htbp]{\scriptsize
\begin{subtable}[t]{0.49\textwidth}
\centering
\begin{tabular}{|l|l|rrrrr|}
\toprule
    & & Bas. & Sch. 1 & Sch. 2 & Sch. 3 & Q-RLT \\
\midrule
Opt & \begin{tabular}{@{}c@{}}\% Diff. \\ Count\end{tabular}  & \begin{tabular}{@{}c@{}}- \\ 10\end{tabular} & \begin{tabular}{@{}c@{}}-161.6 \\ 0\end{tabular} & \begin{tabular}{@{}c@{}}-1.1 \\ 0\end{tabular}  & \begin{tabular}{@{}c@{}} NA \\ 0\end{tabular} & \begin{tabular}{@{}c@{}}-14.0 \\ 0\end{tabular}\\\midrule
Nvars & \begin{tabular}{@{}c@{}}\% Diff. \\ Count\end{tabular}  & \begin{tabular}{@{}c@{}}- \\ 0\end{tabular} & \begin{tabular}{@{}c@{}}-9.9 \\ 0\end{tabular} & \begin{tabular}{@{}c@{}}89.3 \\ 0\end{tabular}  & \begin{tabular}{@{}c@{}} NA \\ 0\end{tabular} & \begin{tabular}{@{}c@{}}-16.7 \\ 30\end{tabular}\\\midrule
Ncons & \begin{tabular}{@{}c@{}}\% Diff. \\ Count\end{tabular}  & \begin{tabular}{@{}c@{}}- \\ 0\end{tabular} & \begin{tabular}{@{}c@{}}5.6 \\ 0\end{tabular} & \begin{tabular}{@{}c@{}}324.2 \\ 8\end{tabular}  & \begin{tabular}{@{}c@{}} NA \\ 0\end{tabular} & \begin{tabular}{@{}c@{}}-6.6 \\ 19\end{tabular}\\
\bottomrule
\end{tabular}
\caption{$\delta=5$, $k=1$}
\label{table:otherdata5k1}
\end{subtable}
\hspace{\fill}
\begin{subtable}[t]{0.49\textwidth}
\centering
\begin{tabular}{|l|l|rrrrr|}
\toprule
    & & Bas. & Sch. 1 & Sch. 2 & Sch. 3 & Q-RLT \\
\midrule
Opt & \begin{tabular}{@{}c@{}}\% Diff. \\ Count\end{tabular}  & \begin{tabular}{@{}c@{}}- \\ 26\end{tabular} & \begin{tabular}{@{}c@{}}-48.8 \\ 0\end{tabular} & \begin{tabular}{@{}c@{}}-9.2 \\ 0\end{tabular}  & \begin{tabular}{@{}c@{}} NA \\ 0\end{tabular} & \begin{tabular}{@{}c@{}}-19.7 \\ 0\end{tabular}\\\midrule
Nvars & \begin{tabular}{@{}c@{}}\% Diff. \\ Count\end{tabular}  & \begin{tabular}{@{}c@{}}- \\ 0\end{tabular} & \begin{tabular}{@{}c@{}}-21.8 \\ 0\end{tabular} & \begin{tabular}{@{}c@{}}245.1 \\ 0\end{tabular}  & \begin{tabular}{@{}c@{}} NA \\ 0\end{tabular} & \begin{tabular}{@{}c@{}}-42.2 \\ 30\end{tabular}\\\midrule
Ncons & \begin{tabular}{@{}c@{}}\% Diff. \\ Count\end{tabular}  & \begin{tabular}{@{}c@{}}- \\ 10\end{tabular} & \begin{tabular}{@{}c@{}}24.0 \\ 0\end{tabular} & \begin{tabular}{@{}c@{}}803.3 \\ 0\end{tabular}  & \begin{tabular}{@{}c@{}} NA \\ 0\end{tabular} & \begin{tabular}{@{}c@{}}-12.6 \\ 20\end{tabular}\\
\bottomrule
\end{tabular}
\caption{$\delta=5$, $k=10$}
\label{table:otherdata5k10}
\end{subtable}

\vspace{0.5cm}
\begin{subtable}[t]{0.49\textwidth}
\centering
\begin{tabular}{|l|l|rrrrr|}
\toprule
    & & Bas. & Sch. 1 & Sch. 2 & Sch. 3 & Q-RLT \\
\midrule
Opt & \begin{tabular}{@{}c@{}}\% Diff. \\ Count\end{tabular}  & \begin{tabular}{@{}c@{}}- \\ 17\end{tabular} & \begin{tabular}{@{}c@{}}-1809.3 \\ 0\end{tabular} & \begin{tabular}{@{}c@{}}-129.7 \\ 0\end{tabular}  & \begin{tabular}{@{}c@{}} NA \\ 0\end{tabular} & \begin{tabular}{@{}c@{}}-155.8 \\ 0\end{tabular}\\\midrule
Nvars & \begin{tabular}{@{}c@{}}\% Diff. \\ Count\end{tabular}  & \begin{tabular}{@{}c@{}}- \\ 0\end{tabular} & \begin{tabular}{@{}c@{}}-76.2 \\ 0\end{tabular} & \begin{tabular}{@{}c@{}}1994.4 \\ 0\end{tabular}  & \begin{tabular}{@{}c@{}} NA \\ 0\end{tabular} & \begin{tabular}{@{}c@{}}-81.9 \\ 30\end{tabular}\\\midrule
Ncons & \begin{tabular}{@{}c@{}}\% Diff. \\ Count\end{tabular}  & \begin{tabular}{@{}c@{}}- \\ 0\end{tabular} & \begin{tabular}{@{}c@{}}-52.8 \\ 0\end{tabular} & \begin{tabular}{@{}c@{}}9821.8 \\ 0\end{tabular}  & \begin{tabular}{@{}c@{}} NA \\ 0\end{tabular} & \begin{tabular}{@{}c@{}}-65.4 \\ 30\end{tabular}\\
\bottomrule
\end{tabular}
\caption{$\delta=10$, $k=1$}
\label{table:otherdata10k1}
\end{subtable}
\hspace{\fill}
\begin{subtable}[t]{0.49\textwidth}
\centering
\begin{tabular}{|l|l|rrrrr|}
\toprule
    & & Bas. & Sch. 1 & Sch. 2 & Sch. 3 & Q-RLT \\
\midrule
Opt & \begin{tabular}{@{}c@{}}\% Diff. \\ Count\end{tabular}  & \begin{tabular}{@{}c@{}}- \\ 29\end{tabular} & \begin{tabular}{@{}c@{}}-44.8 \\ 0\end{tabular}  & \begin{tabular}{@{}c@{}} NA \\ 0\end{tabular}  & \begin{tabular}{@{}c@{}} NA \\ 0\end{tabular} & \begin{tabular}{@{}c@{}}-12.3 \\ 0\end{tabular}\\\midrule
Nvars & \begin{tabular}{@{}c@{}}\% Diff. \\ Count\end{tabular}  & \begin{tabular}{@{}c@{}}- \\ 0\end{tabular} & \begin{tabular}{@{}c@{}}-89.8 \\ 0\end{tabular}  & \begin{tabular}{@{}c@{}} NA \\ 0\end{tabular}  & \begin{tabular}{@{}c@{}} NA \\ 0\end{tabular} & \begin{tabular}{@{}c@{}}-95.2 \\ 30\end{tabular}\\\midrule
Ncons & \begin{tabular}{@{}c@{}}\% Diff. \\ Count\end{tabular}  & \begin{tabular}{@{}c@{}}- \\ 0\end{tabular} & \begin{tabular}{@{}c@{}}-74.8 \\ 0\end{tabular}  & \begin{tabular}{@{}c@{}} NA \\ 0\end{tabular}  & \begin{tabular}{@{}c@{}} NA \\ 0\end{tabular} & \begin{tabular}{@{}c@{}}-87.4 \\ 30\end{tabular}\\
\bottomrule
\end{tabular}
\caption{$\delta=10$, $k=10$}
\label{table:otherdata10k10}
\end{subtable}

\vspace{0.5cm}
\begin{subtable}[t]{0.49\textwidth}
\centering
\begin{tabular}{|l|l|rrrrr|}
\toprule
    & & Bas. & Sch. 1 & Sch. 2 & Sch. 3 & Q-RLT \\
\midrule
Opt & \begin{tabular}{@{}c@{}}\% Diff. \\ Count\end{tabular}  & \begin{tabular}{@{}c@{}} NA \\ 0\end{tabular}  & \begin{tabular}{@{}c@{}}- \\ 0\end{tabular}  & \begin{tabular}{@{}c@{}} NA \\ 0\end{tabular}  & \begin{tabular}{@{}c@{}} NA \\ 0\end{tabular} & \begin{tabular}{@{}c@{}}30.4 \\ 29\end{tabular}\\\midrule
Nvars & \begin{tabular}{@{}c@{}}\% Diff. \\ Count\end{tabular}  & \begin{tabular}{@{}c@{}} NA \\ 0\end{tabular}  & \begin{tabular}{@{}c@{}}- \\ 0\end{tabular}  & \begin{tabular}{@{}c@{}} NA \\ 0\end{tabular}  & \begin{tabular}{@{}c@{}} NA \\ 0\end{tabular} & \begin{tabular}{@{}c@{}}-45.3 \\ 30\end{tabular}\\\midrule
Ncons & \begin{tabular}{@{}c@{}}\% Diff. \\ Count\end{tabular}  & \begin{tabular}{@{}c@{}} NA \\ 0\end{tabular}  & \begin{tabular}{@{}c@{}}- \\ 0\end{tabular}  & \begin{tabular}{@{}c@{}} NA \\ 0\end{tabular}  & \begin{tabular}{@{}c@{}} NA \\ 0\end{tabular} & \begin{tabular}{@{}c@{}}-46.8 \\ 30\end{tabular}\\
\bottomrule
\end{tabular}
\caption{$\delta=15$, $k=1$}
\label{table:otherdata15k1}
\end{subtable}
\hspace{\fill}
\begin{subtable}[t]{0.49\textwidth}
\centering
\begin{tabular}{|l|l|rrrrr|}
\toprule
    & & Bas. & Sch. 1 & Sch. 2 & Sch. 3 & Q-RLT \\
\midrule
Opt & \begin{tabular}{@{}c@{}}\% Diff. \\ Count\end{tabular}  & \begin{tabular}{@{}c@{}} NA \\ 0\end{tabular}  & \begin{tabular}{@{}c@{}}- \\ 0\end{tabular}  & \begin{tabular}{@{}c@{}} NA \\ 0\end{tabular}  & \begin{tabular}{@{}c@{}} NA \\ 0\end{tabular} & \begin{tabular}{@{}c@{}}25.1 \\ 30\end{tabular}\\\midrule
Nvars & \begin{tabular}{@{}c@{}}\% Diff. \\ Count\end{tabular}  & \begin{tabular}{@{}c@{}} NA \\ 0\end{tabular}  & \begin{tabular}{@{}c@{}}- \\ 0\end{tabular}  & \begin{tabular}{@{}c@{}} NA \\ 0\end{tabular}  & \begin{tabular}{@{}c@{}} NA \\ 0\end{tabular} & \begin{tabular}{@{}c@{}}-69.6\\ 30\end{tabular}\\\midrule
Ncons & \begin{tabular}{@{}c@{}}\% Diff. \\ Count\end{tabular}  & \begin{tabular}{@{}c@{}} NA \\ 0\end{tabular}  & \begin{tabular}{@{}c@{}}- \\ 0\end{tabular}  & \begin{tabular}{@{}c@{}} NA \\ 0\end{tabular}  & \begin{tabular}{@{}c@{}} NA \\ 0\end{tabular} & \begin{tabular}{@{}c@{}}-68.3 \\ 30\end{tabular}\\
\bottomrule
\end{tabular}
\caption{$\delta=15$, $k=10$}
\label{table:otherdata15k10}
\end{subtable}

\vspace{0.5cm}
\begin{subtable}[t]{0.49\textwidth}
\centering
\begin{tabular}{|l|l|rrrrr|}
\toprule
    & & Bas. & Sch. 1 & Sch. 2 & Sch. 3 & Q-RLT \\
\midrule
Opt & \begin{tabular}{@{}c@{}}\% Diff. \\ Count\end{tabular}  & \begin{tabular}{@{}c@{}} NA \\ 0\end{tabular}  & \begin{tabular}{@{}c@{}}- \\ 0\end{tabular}  & \begin{tabular}{@{}c@{}} NA \\ 0\end{tabular}  & \begin{tabular}{@{}c@{}} NA \\ 0\end{tabular} & \begin{tabular}{@{}c@{}}40.4 \\ 27\end{tabular}\\\midrule
Nvars & \begin{tabular}{@{}c@{}}\% Diff. \\ Count\end{tabular}  & \begin{tabular}{@{}c@{}} NA \\ 0\end{tabular}  & \begin{tabular}{@{}c@{}}- \\ 0\end{tabular}  & \begin{tabular}{@{}c@{}} NA \\ 0\end{tabular}  & \begin{tabular}{@{}c@{}} NA \\ 0\end{tabular} & \begin{tabular}{@{}c@{}}-54.7\\ 30\end{tabular}\\\midrule
Ncons & \begin{tabular}{@{}c@{}}\% Diff. \\ Count\end{tabular}  & \begin{tabular}{@{}c@{}} NA \\ 0\end{tabular}  & \begin{tabular}{@{}c@{}}- \\ 0\end{tabular}  & \begin{tabular}{@{}c@{}} NA \\ 0\end{tabular}  & \begin{tabular}{@{}c@{}} NA \\ 0\end{tabular} & \begin{tabular}{@{}c@{}}-55.0\\ 30\end{tabular}\\
\bottomrule
\end{tabular}
\caption{$\delta=20$, $k=1$}
\label{table:otherdata20k1}
\end{subtable}
\hspace{\fill}
\begin{subtable}[t]{0.49\textwidth}
\centering
\begin{tabular}{|l|l|rrrrr|}
\toprule
    & & Bas. & Sch. 1 & Sch. 2 & Sch. 3 & Q-RLT \\
\midrule
Opt & \begin{tabular}{@{}c@{}}\% Diff. \\ Count\end{tabular}  & \begin{tabular}{@{}c@{}} NA \\ 0\end{tabular}  & \begin{tabular}{@{}c@{}}- \\ 0\end{tabular}  & \begin{tabular}{@{}c@{}} NA \\ 0\end{tabular}  & \begin{tabular}{@{}c@{}} NA \\ 0\end{tabular} & \begin{tabular}{@{}c@{}}27.3 \\ 30\end{tabular}\\\midrule
Nvars & \begin{tabular}{@{}c@{}}\% Diff. \\ Count\end{tabular}  & \begin{tabular}{@{}c@{}} NA \\ 0\end{tabular}  & \begin{tabular}{@{}c@{}}- \\ 0\end{tabular}  & \begin{tabular}{@{}c@{}} NA \\ 0\end{tabular}  & \begin{tabular}{@{}c@{}} NA \\ 0\end{tabular} & \begin{tabular}{@{}c@{}}-77.4\\ 30\end{tabular}\\\midrule
Ncons & \begin{tabular}{@{}c@{}}\% Diff. \\ Count\end{tabular}  & \begin{tabular}{@{}c@{}} NA \\ 0\end{tabular}  & \begin{tabular}{@{}c@{}}- \\ 0\end{tabular}  & \begin{tabular}{@{}c@{}} NA \\ 0\end{tabular}  & \begin{tabular}{@{}c@{}} NA \\ 0\end{tabular} & \begin{tabular}{@{}c@{}}-76.3\\ 30\end{tabular}\\
\bottomrule
\end{tabular}
\caption{$\delta=20$, $k=10$}
\label{table:otherdata20k10}
\end{subtable}
\caption{Root node quadrification results for different $\delta$ and $k$.}
\label{table:appendix}
}\end{table}

\end{document}